\documentclass[11pt]{aart}

\usepackage[letterpaper, hmargin=1in, top=1.1in, bottom=1.4in, footskip=0.7in]{geometry}

\usepackage{titlesec}
\titleformat{\subsection}[runin]{\normalfont\bfseries}{\thesubsection.}{.5em}{}[.]\titlespacing{\subsection}{0pt}{2ex plus .1ex minus .2ex}{.8em}
\titleformat{\subsubsection}[runin]{\normalfont\itshape}{\thesubsubsection.}{.3em}{}[.]\titlespacing{\subsubsection}{0pt}{1ex plus .1ex minus .2ex}{.5em}
\titleformat{\paragraph}[runin]{\normalfont\itshape}{\theparagraph.}{.3em}{}[.]\titlespacing{\paragraph}{0pt}{1ex plus .1ex minus .2ex}{.5em}

\usepackage{amsmath}
\usepackage{amssymb}
\usepackage{amsfonts}
\usepackage{latexsym}
\usepackage{amsthm}
\usepackage{amsxtra}
\usepackage{amscd}
\usepackage{bbm}
\usepackage{mathrsfs}
\usepackage{bm}

\usepackage{graphicx, color}

\definecolor{darkred}{rgb}{0.9,0,0.3}
\definecolor{darkblue}{rgb}{0,0.3,0.9}
\definecolor{orange}{rgb}{0.98, 0.6, 0.01}

\definecolor{vdarkred}{rgb}{0.6,0,0.2}
\definecolor{vdarkblue}{rgb}{0,0.2,0.6}
\usepackage[pdftex, colorlinks, linkcolor=vdarkblue,citecolor=vdarkred]{hyperref}

\usepackage[nottoc,notlof,notlot]{tocbibind}
\usepackage{cite} 

\numberwithin{equation}{section}
\numberwithin{figure}{section}

\theoremstyle{plain} 
\newtheorem{theorem}{Theorem}[section]
\newtheorem*{theorem*}{Theorem}
\newtheorem{lemma}[theorem]{Lemma}
\newtheorem*{lemma*}{Lemma}

\newtheorem*{corollary*}{Corollary}
\newtheorem{proposition}[theorem]{Proposition}
\newtheorem*{proposition*}{Proposition}

\newtheorem*{conjecture*}{Conjecture}

\theoremstyle{definition} 
\newtheorem{definition}[theorem]{Definition}
\newtheorem*{definition*}{Definition}

\newtheorem*{example*}{Example}
\newtheorem{remark}[theorem]{Remark}
\newtheorem*{remark*}{Remark}
\newtheorem{assumption}[theorem]{Assumption}
\newtheorem*{assumption*}{Assumption}

\newcommand{\f}[1]{\boldsymbol{\mathrm{#1}}} 
\newcommand{\bb}{\mathbb} 
\renewcommand{\cal}{\mathcal}

\newcommand{\ul}[1]{\underline{#1} \!\,} 
\newcommand{\ol}[1]{\overline{#1} \!\,} 

\newcommand{\wt}{\widetilde}

\renewcommand{\P}{\mathbb{P}}
\newcommand{\E}{\mathbb{E}}

\newcommand{\R}{\mathbb{R}}
\newcommand{\C}{\mathbb{C}}
\newcommand{\N}{\mathbb{N}}

\newcommand{\ii}{\mathrm{i}}
\newcommand{\dd}{\mathrm{d}}
\newcommand{\col}{\mathrel{\vcenter{\baselineskip0.75ex \lineskiplimit0pt \hbox{.}\hbox{.}}}}
\newcommand*{\deq}{\mathrel{\vcenter{\baselineskip0.65ex \lineskiplimit0pt \hbox{.}\hbox{.}}}=}
\newcommand*{\eqd}{=\mathrel{\vcenter{\baselineskip0.65ex \lineskiplimit0pt \hbox{.}\hbox{.}}}}

\renewcommand{\leq}{\leqslant}
\renewcommand{\le}{\leqslant}
\renewcommand{\geq}{\geqslant}
\renewcommand{\ge}{\geqslant}
\renewcommand{\epsilon}{\varepsilon}

\newcommand{\qq}[1]{[\![{#1}]\!]}

\newcommand{\ind}[1]{\f 1 (#1)}

\newcommand{\pb}[1]{\bigl({#1}\bigr)}
\newcommand{\pB}[1]{\Bigl({#1}\Bigr)}
\newcommand{\pbb}[1]{\biggl({#1}\biggr)}
\newcommand{\pBB}[1]{\Biggl({#1}\Biggr)}

\newcommand{\qb}[1]{\bigl[{#1}\bigr]}

\newcommand{\qbb}[1]{\biggl[{#1}\biggr]}
\newcommand{\qBB}[1]{\Biggl[{#1}\Biggr]}

\newcommand{\h}[1]{\{{#1}\}}
\newcommand{\hb}[1]{\bigl\{{#1}\bigr\}}

\newcommand{\abs}[1]{\lvert #1 \rvert}
\newcommand{\absb}[1]{\bigl\lvert #1 \bigr\rvert}

\newcommand{\norm}[1]{\lVert #1 \rVert}

\newcommand{\scalar}[2]{\langle{#1} \mspace{2mu}, {#2}\rangle}

\DeclareMathOperator{\tr}{Tr}
\DeclareMathOperator{\var}{Var}

\DeclareMathOperator{\re}{Re}
\DeclareMathOperator{\im}{Im}

\title{Isotropic self-consistent equations for mean-field random matrices}

\author{Yukun He\footnote{University of Geneva, Section of Mathematics, {\tt yukun.he@unige.ch}.} \and Antti Knowles\footnote{University of Geneva, Section of Mathematics, {\tt antti.knowles@unige.ch}.} \and Ron Rosenthal\footnote{Technion -  Israel Institute of Technology, Department of Mathematics, {\tt ron.ro@tx.technion.ac.il}.}}

\begin{document}

\maketitle

\begin{abstract}
We present a simple and versatile method for deriving (an)isotropic local laws for general random matrices constructed from independent random variables. Our method is applicable to mean-field random matrices, where all independent variables have comparable variances. It is entirely insensitive to the expectation of the matrix. In this paper we focus on the probabilistic part of the proof -- the derivation of the self-consistent equations. As a concrete application, we settle in complete generality the local law for Wigner matrices with arbitrary expectation.
\end{abstract}

\section{Introduction}

\subsection{Overview}
The empirical eigenvalue measure of a large random matrix is typically well approximated by a deterministic asymptotic measure. For instance, for Wigner matrices this measure is the celebrated Wigner semicircle law \cite{Wig}. This approximation is best formulated using the Green function and Stieltjes transforms. Let $W$ be an $N \times N$ Hermitian random matrix normalized so that its typical eigenvalue spacing is of order $N^{-1}$, and denote by
\begin{equation*}
G(z) \deq (W - z)^{-1}
\end{equation*}
the associated Green function. Here $z = E + \ii \eta$ is a spectral parameter with positive imaginary part $\eta$. Then the Stieltjes transform of the empirical eigenvalue measure is equal to $N^{-1} \tr G(z)$, and the approximation mentioned above may be written informally as
\begin{equation} \label{intr_avg_ll}
\frac{1}{N} \tr G(z) \approx m(z)
\end{equation}
for large $N$ and with high probability. Here $m(z)$ is the Stieltjes transform of the asymptotic eigenvalue measure, which we denote by $\varrho$. We call an estimate of the form \eqref{intr_avg_ll} an \emph{averaged law}.

As may be easily seen by taking the imaginary part of \eqref{intr_avg_ll}, control of the convergence of $N^{-1} \tr G(z)$ yields control of an order $\eta N$ eigenvalues around the point $E$. A \emph{local law} is an estimate of the form \eqref{intr_avg_ll} for all $\eta \gg N^{-1}$. Note that the approximation \eqref{intr_avg_ll} cannot be correct at or below the scale $\eta \asymp N^{-1}$, at which the behaviour of the left-hand side of \eqref{intr_avg_ll} is governed by the fluctuations of individual eigenvalues.

Local laws have become a cornerstone of random matrix theory, starting from the work \cite{ESY2} where a local law was first established for Wigner matrices. Local laws constitute the main tool needed to analyse the distribution of the eigenvalues of random matrices, including the universality of the local spectral statistics and the proof of the celebrated Wigner-Dyson-Mehta conjecture \cite{Meh}, and the eigenvectors of random matrices, including eigenvector delocalization and the distribution of the eigenvector components.

In fact, for the aforementioned applications to the distribution of eigenvalues and eigenvectors, the averaged local law from \eqref{intr_avg_ll} is not sufficient. One has to control not only the normalized trace of $G$ but the matrix $G$ itself, by showing that $G$ is close to some deterministic matrix, $M$, provided that $\eta \gg N^{-1}$. As the dimension $N$ of these matrices is very large, which notion of closeness to use is a nontrivial question. The control of $G - M$ as a matrix was initiated in \cite{EYY1} in the context of Wigner matrices, where the individual matrix entries $G_{ij} - M_{ij}$ are controlled. We call such a result an \emph{entrywise local law}. Subsequently, the entrywise local law of \cite{EYY1} has been refined and significantly generalized \cite{PY, AEK,AEK1,AEK2,AEK3,AEK4,Che,BES15,EYY2,EYY3,KY2,EKYY1,LeeSch,LSSY,EKYY4,BEKYY,EKYY3,KY4,BKY15,LS1}.

More generally, canonical notions are closeness in the \emph{weak operator sense}, \emph{strong operator sense}, and \emph{norm sense}, which entail control of
\begin{equation*}
\scalar{\f w}{(G - M) \f v} \,, \qquad
\abs{(G - M) \f v}\,, \qquad \norm{G - M}
\end{equation*}
respectively, for deterministic $\f v, \f w \in \C^N$. It is easy to see that already convergence in the strong operator sense must fail, as one typically has $\abs{(G - M)\f v} \asymp \frac{1}{\sqrt{N} \eta}$; we refer to \cite[Section 12.1]{BK16} for more details. We conclude that control in the weak operator sense, i.e.\  of  $\scalar{\f w}{(G - M) \f v}$, is the strongest form of control that one may expect to hold for the error $G - M$. Following \cite{KY2, BEKYY, KY4} we call such a result an \emph{(an)isotropic local law}. (An)isotropic local laws have played a central role in the study of deformed matrix ensembles \cite{KY2, BEKYY, KY4, KYY} and the distribution of eigenvector components of random matrices \cite{BoY1, KYY, BHY}.

All proofs of local laws consist of at least the two following major steps, which are completely decoupled.
\begin{enumerate}
\item[(A)]
\emph{A stochastic step} establishing a self-consistent equation for $G$ with a random error term.
\item[(B)]
\emph{A deterministic step} analysing the stability of the self-consistent equation.
\end{enumerate}
Step (A) may be formulated in general as follows. There is a map $\Pi \col \C^{N \times N} \to \C^{N \times N}$ such that $M$ is uniquely charaterized as the solution of the equation $\Pi(M) = 0$ with positive imaginary part. For a large class of random matrices, including those considered in this paper, $\Pi$ is a nonlinear quadratic function. Step (A) provides high-probability bounds on the matrix $\Pi(G)$ (in some suitable sense).

In Step (B) one shows that if $\Pi(G) = \Pi(G) - \Pi(M)$ is small then so is $G - M$.

In general, both steps are nontrivial. In addition, steps (A) and (B) have up to now only been used to derive an entrywise local law; to obtain a full (an)isotropic local law, a further nontrivial argument was necessary; three different techniques for deriving (an)isotropic laws from entrywise laws were developed in \cite{KY2,BEKYY,KY4}.

In this paper we develop a new method for dealing with Step (A), which has the following advantages as compared to previous approaches.
\begin{enumerate}
\item
It can handle very general mean-field random matrices $W$, and is in particular completely insensitive to the expectation $\E W$. (Previous methods heavily rely on the assumption that $\E W$ is diagonal or close to diagonal.)
\item
It automatically yields the (an)isotropic law in one step, without requiring a further intermediate step via an entrywise law, as in \cite{KY2,BEKYY,KY4}.
\item
It is simple and very versatile. It is applicable to many matrix models where previous methods either fail or are very cumbersome.
\end{enumerate}

For conciseness, in this paper we focus on mean-field Hermitian random matrices whose upper-triangular entries are independent. By \emph{mean-field} we mean that $\var (W_{ij}) = O(N^{-1})$ for all $i,j$. We make no assumption on $\E W$.
We remark that our method extends to very general matrices built from independent random variables, such as sample covariance matrices built from populations with arbitrary expectations and covariances.

We now outline our main results. Throughout the paper we split
\begin{equation*}
W = H + A \,, \qquad A \deq \E W\,.
\end{equation*}
We define the map
\begin{equation} \label{Pi_intro}
\Pi(M) \deq I + zM + \cal S(M) M  - AM \,, \qquad \cal S(M) \deq \E [HMH]\,.
\end{equation}
Then it is not hard to show that for $z \in \C_+$, the equation $\Pi(M) = 0$ has a unique solution $M$ with positive imaginary part (see Lemma \ref{lem:sol_Pi} below). Our main result (Theorem \ref{thm:main_general} below) is the derivation of the self-consistent equation $\Pi(G) \approx 0$: we establish high-probability bounds on the quantities
\begin{equation*}
\scalar{\f v}{\Pi(G) \f w}\,, \qquad \tr [B \Pi(G)] \,,
\end{equation*}
where $B$ is a bounded deterministic matrix and $\f v , \f w \in \C^N$ are deterministic.
These bounds are sharp (up to some technicalities in the definition of a high probability bound) throughout the spectrum, in both the bulk and in the vicinity of regular edges where the density of the limiting measure $\varrho$ exhibits the characteristic square-root decay of random matrix theory. We emphasize that we make no assumption on the expectation $A = \E W$.

Our main result performs Step (A) for the general class of mean-field random matrices introduced above. To illustrate how this result may be combined with Step (B) to complete the proof of the (an)isotropic local law, we use our main result to settle in complete generality the local law for Wigner matrices with arbitrary expectation (Theorem \ref{thm:Wigner} below), i.e.\ $\E \abs{H_{ij}}^2 = N^{-1}$ with arbitrary expectation $A$. Previously, the local law for Wigner matrices with nonzero expectation was derived under the assumption that $A$ is diagonal \cite{LSSY} or that all third moments of $H$ vanish \cite{KY4}. Even the averaged local law for general $A$ was not accessible using previous methods.

In this paper we do not perform Step (B) (the stability analysis of the self-consistent equation $\Pi(G) \approx 0$) for the aforementioned general class of mean-field matrices. For general $W$ this is a nontrivial task, and we refer to the works \cite{AEK1, AEK2, AEK3,AEK,AEK4}, where a related analysis is performed to obtain an entrywise law for the case of almost diagonal $A$.

We remark that the map $\Pi$ defined in \eqref{Pi_intro} has been extensively studied in the random matrix literature, starting with \cite{G01, PS11}, where its stability was analysed on the global spectral scale. Recently, it has been the focus of the seminal series of works \cite{AEK1, AEK2, AEK3,AEK,AEK4} mentioned above, where the stability of the matrix equation $\Pi(M) = 0$ was analysed in great detail in order to establish local laws for a general class of random matrix ensembles.

We conclude this introductory section with a few words about the proof. Traditionally, in all of the works \cite{PY, EYY1,AEK,AEK1,AEK2,AEK3,AEK4,Che,BES15,EYY2,EYY3,KY2,EKYY1,LeeSch,LSSY,EKYY4,BEKYY,EKYY3,KY4}, the proof of Step (A) for matrices with independent entries relied on Schur's complement formula. When applicable, Schur's complement formula provides a direct approach to proving Step (A). However, its applicability to random matrix models is heavily dependent on the independence of the entries of $H$ and on the fact that $A$ is diagonal. If the mapping between the independent random variables and the entries of $W$ is nontrivial or if $A$ is not diagonal, the use of Schur's complement formula becomes more cumbersome and possibly even fruitless. Moreover, Schur's complement formula is only effective when deriving entrywise estimates, and it is ill-suited for deriving isotropic estimates. The first proof of a local law without using Schur's complement formula was \cite{BKY15}, where a local law was established for random regular graphs (which are not built from independent random variables). In \cite{BKY15}, Schur's complement formula was replaced by a much more robust expansion using the resolvent identity. Our current approach is in part motivated by this philosophy: instead of working with entire rows and columns, as dictated by Schur's complement formula, we work with individual matrix entries. A natural way of achieving this is to perform a series of expansions using the resolvent identity, where $G$ is expanded in a single entry of $H$. We remark that such ideas have previously been used in the somewhat different context of proving the so-called fluctuation averaging result in random matrix theory; see \cite{EYY3}.

In practice, we choose to replace the resolvent expansion with a cumulant expansion inspired by \cite{KKP, HK, LSY} (see Lemma \ref{lem:cumulant_expansion} below), which effectively performs the same steps but in a more streamlined fashion. A drawback of the cumulant expansion is that it requires working with expectations of random variables instead of just random variables. This leads us to consider high moments of error terms, which are estimated self-consistently. An analogous strategy was first used in \cite{KY2} to prove the isotropic local law for Wigner matrices, although there the moments were estimated by the Green function comparison strategy instead of the cumulant expansion used in this paper. The first use of recursive self-consistent estimates for high moments using the cumulant expansion, including exploiting a cancellation among two leading terms, was \cite{HK}; this strategy constitutes the core of our method. The cumulant expansion was first used to derive a local law in the recent work \cite{LS1}, where a precise averaged law was derived for sparse random matrices with $A = 0$. Analogous ideas for the Haar measure on the unitary group were recently exploited in \cite{BLS} to obtain optimal bounds on the convergence rate for the spectral distribution of the sum of random matrices. Historically, the first use of the cumulant expansion for local spectral statistics is \cite{LSY}, where edge universality and the Tracy-Widom limit for deformed Wigner matrices was proved, along with a streamlined proof of the edge universality for Wigner matrices.

Aside from the novel idea of using the resolvent/cumulant expansion to prove (an)isotropic local laws for random matrices with independent entries, the key new ingredients of our method may be summarized as follows.
\begin{enumerate}
\item
We bootstrap isotropic bounds on the Green function, i.e.\ our proof makes use of a priori bounds of the form $\scalar{\f v}{G \f w} \lesssim \abs{\f v}\abs{\f w}$ with high probability. When using the cumulant expansion (or, alternatively, the resolvent expansion), we are naturally led to sums of the form $\sum_{i} v_i G_{ij} \cdots$. It is crucial to first perform the sum over $i$ to exploit the isotropic a priori bounds on $G$; simply estimating the sum term by term leads to errors which are not affordable. Reducing such sums to isotropic estimates on the Green function requires a certain amount of care.
\item
A direct application of the cumulant expansion to estimating high moments of the error results in terms that cannot be controlled with a high enough precision. To circumvent this issue, we use multiple self-improving matrix bounds at several stages of the proof. We assume a rough a priori bound on an error matrix and derive a self-improving bound on it, which may be iterated to obtain the optimal bound. See e.g.\ \eqref{D_self_improv} or \eqref{Q_self_improv} below for an example. For the estimate of $\tr [B \Pi(G)]$, we apply this approach to a matrix that we call $Q$, defined in \eqref{eq:defn_of_Q} below. The identification of $Q$ as a central object of the proof and the derivation of optimal bounds on it by means of a self-improving scheme obtained from a second cumulant expansion is a key idea of our proof.
\end{enumerate}

\subsubsection*{Conventions}
The central objects of this paper, such as the Green function $G$, depend on $N \in \N$ and the spectral parameter $z \in \C_+ \deq \{z \in \C \col \im z > 0\}$. We frequently omit both arguments, writing for instance $G$ instead of $G_N(z)$. If some quantity does not depend on $N$, we indicate this explicitly by calling it \emph{constant} or \emph{fixed}. We always use the notation $z = E + \ii \eta$ for the real and imaginary parts of $z$. We use the usual $O(\cdot)$ notation, and write $O_\alpha(\cdot)$ if the implicit constant depends on the parameter $\alpha$. The parameter $\alpha$ can never depend on $N$ or $z$.

We use the notation $\f v = (v_i)_{1\leq i \leq N} \in \C^N$ for vectors, $\langle\cdot,\cdot\rangle$ for the scalar product $\scalar{\f v}{\f w} \deq \sum_{i=1}^N \ol v_i w_i$ and $|\f v|=\sqrt{\langle \f v,\f v\rangle}$ for the Euclidean norm. We denote by $\norm{B}$ the Euclidean operator norm of an $N \times N$ matrix $B$. We write $\ul B \deq \frac{1}{N} \tr B$ for the normalized trace of $B$.

\subsubsection*{Acknowledgements}
We thank Torben Kr\"uger for drawing our attention to the importance of the equation \eqref{def_M} for general mean-field matrix models.
In a private communication, Torben Kr\"uger also informed us that the idea of organizing proofs of local laws by bootstrapping is being developed independently for random matrices with general correlated entries in \cite{EK}.

\subsection{Basic definitions}

We consider $N \times N$ matrices of the form $W = H + A \in \C^{N \times N}$, where $H = H^*$ is random and $A = A^*$ is deterministic. We always make the following assumption on $H$.
\begin{assumption}\label{ass:H}
	The upper-triangular entries $(H_{ij} \col 1\leq i\leq j\leq N)$ are independent mean-zero random variables satisfying $E[ \abs{\sqrt{N} H_{ij}}^p] = O_p(1)$ for all $i,j$ and $p \in \N$.
\end{assumption}

Let $z = E + \ii \eta \in \C_+$ be a spectral parameter, and define the Green function
\begin{equation*}
	G(z) \deq (H+A-z)^{-1}\,.
\end{equation*}
We use the following notion of high-probability bounds, the first version of which appeared in \cite{EKY2}.
\begin{definition}[Stochastic domination]\label{def:stocdom}$~$
\begin{enumerate}
	\item Let
		\begin{equation*}
			X = \pb{X^{(N)}(u) \col N \in \N, u \in U^{(N)}} \,, \qquad
			Y = \pb{Y^{(N)}(u) \col N \in \N, u \in U^{(N)}}
		\end{equation*}
		be two families of nonnegative random variables, where $U^{(N)}$ is a possibly 	$N$-dependent parameter set. 
		
		We say that $X$ is \emph{stochastically dominated by $Y$, uniformly in $u$,} if for all (small) $\epsilon > 0$ and (large) $D > 0$ we have
		\begin{equation*}
			\sup_{u \in U^{(N)}} \P \pB{X^{(N)}(u) > N^\epsilon Y^{(N)}(u)} \leq N^{-D}
		\end{equation*}
		for large enough $N\geq N_0(\epsilon, D)$. Throughout this paper the stochastic domination will always be uniform in all parameters (such as matrix indices, deterministic vectors, and spectral parameters $z$) that are not explicitly fixed. Note that $N_0(\epsilon, D)$ may depend on quantities that are explicitly constant.
	\item If $X$ is stochastically dominated by $Y$, uniformly in $u$, we use the notation $X \prec Y$. Moreover, if for some complex family $X$ we have $\abs{X} \prec Y$ we also write $X = O_\prec(Y)$. (Note that for deterministic $X$ and $Y$, $X=O_\prec(Y)$ means $X=O_\epsilon(N^{\varepsilon}Y)$ for any $\varepsilon>0$.)
	\item We extend the definition of $O_\prec(\,\cdot\,)$ to matrices in the weak operator sense as follows. Let $X$ be a family of complex $N \times N$ random matrices and $Y$ a family of nonnegative random variables. Then we write $X = O_\prec(Y)$ to mean $\abs{\scalar{\f v}{X \f w}} \prec \abs{\f v} \abs{\f w} Y$ uniformly for all deterministic vectors $\f v, \f w \in\C^N$.
	\end{enumerate}
\end{definition}

\subsection{Derivation of the self-consistent equation}

We use the notation $A \geq 0$ to mean that $A \in \C^{N \times N}$ is a positive semidefinite matrix, i.e.\ it is Hermitian and all of its eigenvalues are nonnegative.

\begin{definition} \label{def_S_F}
Define $\im M \deq \frac{1}{2 \ii} (M - M^*)$ and $\cal M_+ \deq \hb{M \in \C^{N \times N} \col \im M \geq 0}$. Let $\cal S \col \C^{N \times N} \to \C^{N \times N}$ be a linear map that maps $\cal M_+$ to itself. For $z \in \C_+$ define the function $\Pi \equiv \Pi(\cdot,z) \col \C^{N \times N} \to \C^{N \times N}$ through
\begin{equation*}
\Pi(M) \equiv \Pi(M,z) \deq I + zM + \cal S(M)M - AM\,.
\end{equation*}
\end{definition}

The following deterministic result was proved in \cite{HFS}.
\begin{lemma} \label{lem:sol_Pi}
Let $\cal S$ and $\Pi$ be as in Definition \ref{def_S_F}. Then for any $z \in \C_+$ the equation
\begin{equation} \label{def_M}
\Pi(M) = 0
\end{equation}
has a unique solution $M \in \cal M_+$. We denote this solution by $M(z)$.
\end{lemma}

We now choose $\cal S$ to be
\begin{equation} \label{def_S_gen}
\cal S(M) \deq \E [H M H]\,.
\end{equation}
Fix a constant $\tau > 0$ and define the fundamental domain
\begin{equation} \label{def_S}
	\f D \equiv \f D_N(\tau) \deq \hb{E+\mathrm{i}\eta\col |E|\le \tau^{-1}\,,\, N^{-1+\tau}\le \eta\le \tau^{-1}}\,.
\end{equation}
Our first main result gives optimal high-probability bounds on $\Pi(G)$.

\begin{theorem}[Self-consistent equation for $G$] \label{thm:main_general}
Suppose that Assumption \ref{ass:H} holds. Denote by $M \in \cal M_+$ the solution of \eqref{def_M} with respect to the linear map \eqref{def_S_gen}. Let $z \in \f D$ and suppose that $\norm{M} = O(1)$ and $G - M = O_\prec(\phi)$ for some deterministic $\phi \in [N^{-1}, N^{\tau/10}]$ hold at $z$. Then the following estimates hold at $z$.
\begin{enumerate}
\item
We have
\begin{equation*}
\Pi(G) = O_\prec \pBB{(1 + \phi)^3 \sqrt{\frac{\|\im M\| + \phi+\eta}{N \eta}}\,}\,.
\end{equation*}
\item
For any deterministic $B \in \C^{N \times N}$ satisfying $\norm{B} = O(1)$ we have
\begin{equation*}
	\underline{B \Pi(G)}=O_{\prec}\bigg((1+\phi)^6 \frac{\|\im M\| + \phi+\eta}{N \eta}\bigg)\,.
\end{equation*}
\end{enumerate}
\end{theorem}

\begin{remark}
The expression $\ul {B \Pi(G)}$ is the most general linear function of $\Pi(G)$, in the sense that any linear function $\Phi \col \C^{N \times N} \to \C$ can be written in the form $\Phi(\Pi) = \ul {B \Pi}$ for some $B \in \C^{N \times N}$.
\end{remark}

\subsection{Application to Wigner matrices}

We make the following assumption about Wigner matrices.
\begin{assumption} \label{ass:Wigner}
	For all $1\leq i,j\leq N$ we have $\E |H_{ij}|^2=N^{-1} (1 + O(\delta_{ij}))$.
\end{assumption}

For Wigner matrices satisfying Assumption \ref{ass:Wigner}, instead of \eqref{def_S_gen} it is more convenient to use the slightly different expression
\begin{equation} \label{def_S_wig}
\cal S(M) \deq m I\,, \qquad m \deq \ul M\,.
\end{equation}
It is immediate that, under the Assumption $\norm{M} = O(1)$, the entries of the quantities \eqref{def_S_gen} and \eqref{def_S_wig} differ by $O(N^{-1})$.

The equations \eqref{def_M} and \eqref{def_S_wig} have been studied in detail in the literature on deformed Wigner matrices. They may be rewritten as
\begin{equation*}
M = (A - z - m)^{-1}\,,
\end{equation*}
where $m$ solves the equation
\begin{equation} \label{eq_spectr_A}
m = \int \frac{\nu(\dd a)}{a - m - z}
\end{equation}
and
\begin{equation} \label{def_nu}
\nu \deq \frac{1}{N} \sum_{a \in \mathrm{spec}(A)} \delta_{a}
\end{equation}
is the empirical spectral measure of $A$. The function $m$ is a holomorphic function from $\C_+$ to $\C_+$, and we find immediately that $z m(z) \to -1$ as $z \to \infty$ in $\C_+$. By the integral representation of Nevanlinna functions, we conclude that $m$ is the Stieltjes transform of a probability measure $\varrho$ on $\R$:
\begin{equation} \label{m_varrho}
m(z) = \int \frac{\varrho(\dd x)}{x - z}\,.
\end{equation}
We remark that the measure $\varrho$ is the \emph{free additive convolution} of $\nu$ and the semicircle law $\mu$, written $\varrho=\mu \boxplus \nu$. Fore more details, we refer to \cite{VDN} for a review of free probability theory. For recent results on the local behaviour of the free additive convolution, we refer to \cite{LeeSch,LSSY}.

Define the Stieltjes transform of the empirical spectral measure of $H+A$ as
\begin{equation*}
g \deq \ul G\,.
\end{equation*}
\begin{definition}
We call a subset $\f S \equiv \f S_N \subset \f D$ a \emph{spectral domain} if for each $z \in \f S$ we have
\begin{equation*}
\h{w \in \f D \col \re w = \re z, \im w \geq \im z} \subset \f S\,.
\end{equation*}
\end{definition}

\begin{theorem} \label{thm:Wigner}
Suppose that Assumptions \ref{ass:H} and \ref{ass:Wigner} hold and that $\norm{A} = O(1)$. Let $M \in \cal M_+$ and $m$ be the solutions of \eqref{def_M} and \eqref{def_S_wig}.
Let $\f S$ be a spectral domain. Suppose that $\norm{M} = O(1)$ on $\f S$ and that \eqref{eq_spectr_A} is stable on $\f S$ in the sense of Definition \ref{def:stability} below. Then we have for $z \in \f S$
\begin{equation} \label{G-M_Wig}
G - M = O_\prec\pBB{\sqrt{\frac{\im m}{N \eta}} + \frac{1}{N \eta}}
\end{equation}
and
\begin{equation} \label{g-m_Wig}
g - m = O_\prec\pbb{\frac{1}{N \eta}}\,.
\end{equation}
\end{theorem}

\begin{remark}
The assumptions that $\norm{M} = O(1)$ on $\f S$ and that \eqref{eq_spectr_A} is stable on $\f S$ in the sense of Definition \ref{def:stability} only pertain to the deterministic measure $\nu$ from \eqref{def_nu}, and have been extensively studied in the literature, starting with the work \cite{LeeSch}. For instance, a sufficient condition is that
\begin{equation*}
\inf_{x \in I_\nu} \int \frac{\nu (\dd a)}{(x - a)^2} \geq 1 + c
\end{equation*}
for some constant $c > 0$, where $I_\nu$ is the smallest closed interval containing the support of $\nu$. See \cite{LSSY} for more details.
\end{remark}

\begin{remark}
In Theorem \ref{thm:Wigner}, we make the assumption $\norm{A} = O(1)$ for convenience in order to simplify the analysis of the self-consistent equation \eqref{eq_spectr_A}. This assumption can be easily relaxed; in particular, it is not needed to apply Theorem \ref{thm:main_general}.
\end{remark}

\begin{remark}
A corollary of Theorem \ref{thm:Wigner} is that the Wigner-Dyson-Mehta conjecture holds for Wigner matrices with arbitrary expectation. In other words, $H+A$ satisfying the assumptions of Theorem \ref{thm:Wigner} has universal (sine kernel) local spectral statistics. This follows immediately by combining Theorem \ref{thm:Wigner} with \cite{LSSY}.
\end{remark}

\section{Preliminaries}

The rest of the paper is devoted to the proofs of Theorems \ref{thm:main_general} and \ref{thm:Wigner}. To simplify notation, from now on we only consider the case where all matrix and vector entries are real. The complex case is a simple extension, whose details are given in Appendix \ref{sec:complex}.

In this section we collect notation and various tools which are used throughout the paper. 

\paragraph{Additional notation}

Let  $\qq{n} \deq \{1, 2, \dots, n\}$ and denote by $\bb S \deq \h{\f v \in \R^N \col \abs{\f v} = 1}$ the unit sphere of $\R^N$. For an $N \times N$ matrix $X \in \R^{N \times N}$, vectors $\f v, \f w \in \R^N$, and $i,j \in \qq{N}$, we use the notations
\begin{equation*}
	X_{\f v \f w} \deq \scalar{\f vX}{ \f w}\,, \qquad X_{i \f w} \deq \scalar{\f e_iX}{ \f w}\,, \qquad X_{\f v j} \deq \scalar{\f vX}{ \f e_j} \,, \qquad
X_{ij} \deq \scalar{\f e_iX}{ \f e_j}\,, 
\end{equation*}
where $\f e_i$ is the standard $i$-th basis vector of $\R^N$. We use the abbreviation $G^k_{\f v \f w} \deq (G_{\f v \f w})^k$ for $k \in \N$.

\paragraph{Basic properties of stochastic domination}

The following lemma collects some of the basic properties of stochastic domination. We use the lemma tacitly throughout the paper.

\begin{lemma}[Basic properties of $\prec$]\label{lem:basic_properties_of_prec} $~$
\begin{enumerate}
	\item Suppose that $X(v) \prec Y (v)$ for all $v\in V$. If $|V|\leq N^C$ for some constant $C$ then $\sum_{v\in V}X(v)\prec \sum_{v\in V}Y(v)$.
	\item Suppose that $X_1 \prec Y_1$ and $X_2 \prec Y_2$. Then $X_1 X_2 \prec Y_1 Y_2$.
	\item Suppose that $X \le N^C$ and $Y \ge N^{-C}$ for some constant $C>0 $. Then $X \prec Y$ implies $\E[X] \prec \E[Y]$.
\end{enumerate}
If the above random variables depend on some parameter $u$ and the hypotheses are uniform in $u$ then so are the conclusions.
\end{lemma}

\begin{proof}
	The proof follows from the definition of stochastic domination together with a union bound argument. 
\end{proof}

\begin{lemma} \label{lem:H_bound}
Under Assumption \ref{ass:H} we have $\norm{H} \prec 1$.
\end{lemma}
\begin{proof}
This is a standard application of the F\"uredi-Koml\'os argument (see e.g.\ \cite[Section 2.1.6]{AGZ}).
\end{proof}

\paragraph{Ward identity} An important ingredient of the proofs is the following well-known identity.

\begin{lemma}[Ward identity] \label{lem:Ward_identity} 
	For every vector $\f x\in \R^N$
	\begin{equation*}
		\sum_j |G_{\f x j}|^2 = \frac{\im G_{\f x\f x}}{\eta}\,.	
	\end{equation*}
\end{lemma}

\begin{proof}
This follows immediately from the resolvent identity $G - G^* = 2 \ii \eta GG^*$.
\end{proof}
As a consequence, for any $B \in \R^{N \times N}$ we have the estimate
\begin{equation} \label{ward_estimate}
\sum_{j} \abs{(GB)_{\f x j}}^2 = (GBB^*G^*)_{\f x \f x} \leq \norm{B}^2 (GG^*)_{\f x \f x} = \norm{B}^2 \frac{\im G_{\f  x\f x}}{\eta}\,.
\end{equation}

\paragraph{Cumulant expansion}

Recall that for a real random variable $h$, all of whose moments are finite, the $k$-cumulant of $h$ is
\begin{equation*}
	\cal C_k(h)\deq(-\mathrm{i})^k\bigg(\frac{\dd^{k}}{\dd t^k}\log\mathbb{E}[e^{\mathrm{i}th}]\bigg)\Bigg|_{t=0}.
\end{equation*}
We shall use a standard cumulant expansion from \cite{KKP, Kho2, HK}, which we record in Lemma \ref{lem:cumulant_expansion} below. We require slightly stronger bounds on the remainder term than in previous applications, and for completeness we give a proof of Lemma \ref{lem:cumulant_expansion} in Appendix \ref{A}.

\begin{lemma}[Cumulant expansion]\label{lem:cumulant_expansion}
	Let $f:\R\to\C$ be a smooth function, and denote by $f^{(k)}$ its $k$-th derivative. Then, for every fixed $\ell \in\N$, we have 
\begin{equation}\label{eq:cumulant_expansion}
	\mathbb{E}\big[h\cdot f(h)\big]=\sum_{k=0}^{\ell}\frac{1}{k!}\mathcal{C}_{k+1}(h)\mathbb{E}[f^{(k)}(h)]+\cal R_{\ell+1},
\end{equation}	
 assuming that all expectations in \eqref{eq:cumulant_expansion} exist, where $\cal R_{\ell+1}$ is a remainder term (depending on $f$ and $h$), such that for any $t>0$,
\begin{equation} \label{R_l+1}
\cal R_{\ell+1} = O(1) \cdot \bigg(\E\sup_{|x| \le |h|} \big|f^{(\ell+1)}(x)\big|^2 \cdot \E \,\big| h^{2\ell+4} \mathbf{1}_{|h|>t} \big| \bigg)^{1/2} +O(1) \cdot \bb E |h|^{\ell+2} \cdot  \sup_{|x| \le t}\big|f^{(\ell+1)}(x)\big|\,.
\end{equation}
\end{lemma}

The following result gives bounds on the cumulants of the entries of $H$.
\begin{lemma}\label{lem:cumulant_factos_estimate} If $H$ satisfies Assumption \ref{ass:H} then for every $i,j\in\qq{N}$ and $k\in \N$ we have
	\begin{equation*}
		\cal C_{k+1}(H_{ij})=O_k\big(N^{-(k+1)/2}\big)
	\end{equation*}
	and $\cal C_{1}(H_{ij})=0$.
\end{lemma}

\begin{proof} This follows easily by the homogeneity of the cumulants.
\end{proof}

\subsubsection*{Self-improving high-probability bounds}
Throughout the proof, we shall repeatedly use self-improving high-probability bounds, of which the following lemma is a prototype.
\begin{lemma} \label{lem:self-improv}
	Let $C > 0$ be a constant, $X \geq 0$, and $Y \in [N^{-C}, N^C]$. Suppose there exists a constant $q \in [0,1)$ such that for any $Z \in [Y, N^C]$, we have the implication
	\begin{equation} \label{self-improving-intro}
	X \prec Z \qquad \Longrightarrow \qquad X \prec Z^qY^{1-q}\,,
	\end{equation}
	Then we have $X \prec Y$ provided that $X \prec N^C$.
\end{lemma}
\begin{proof}
	By iteration of \eqref{self-improving-intro} $k$ times, starting from $X \prec N^C$, we find
	\begin{equation*}
	X \prec N^{C q^k} Y^{1 - q^k} \leq N^{2C q^{k}} Y\,.
	\end{equation*}
	Since $2Cq^k$ can be made arbitrarily small by choosing $k$ large enough, the claim follows.
\end{proof}

\section{
	Proof of Theorem \ref{thm:main_general}}

\subsection{Preliminary estimates}\label{subsec:preliminarties_to_proofs}

In this subsection we collect some tools and basic estimates that are used in the proof of Theorem \ref{thm:main_general}. 
Throughout the following we consider two fixed deterministic vectors $\f v,\f w\in \mathbb{S}$.

Define
\begin{equation} \label{31}
s_{ij}\deq (1+\delta_{ij})^{-1} N \bb E \big|H_{ij}\big|^2\,,
\end{equation}
and for a vector $\f x = (x_i)_{i \in \qq{N}} \in \R^N$ denote 
\begin{equation}\label{eq:tilted_vectors}
\f x^{j} = (x^{j}_i)_{i\in \qq{N}} \,, \qquad \text{ where } \quad x^{j}_i \deq x_i {s}_{ij}\,.
\end{equation}
We define the set of vectors
\begin{equation*}
\bb X \equiv \bb X(\f v, \f w) \deq \{\f v, \f w, \f v^{1},\dots,\f v^{N}, \f e_1,\dots,\f e_N\}\,.
\end{equation*}
Note that, due to Assumption \ref{ass:H}, we have $s_{ij} = O(1)$, so that $\abs{\f x^j} = O(\abs{\f x})$ for all $\f x \in \R^N$ and $j \in \qq{N}$, and in particular $|\f x|=O(1)$ for all $\f x \in \bb X$. Furthermore, under the assumptions $\norm{M} = O(1)$ and $G-M =O_{\prec} (\phi)$ of Theorem \ref{thm:main_general} we have, by a union bound,
\begin{equation}\label{eq:bound_on_G}
\max_{\f x,\f y\in \bb X}\abs{G_{\f x\f y}}\leq \max_{\f x,\f y\in \bb X}\abs{M_{\f x\f y}}+\max_{\f x,\f y\in \bb X}\abs{(G-M)_{\f x\f y}}\prec 1+\phi\,. 
\end{equation}

We abbreviate
\begin{equation}\label{eq:psi}
\zeta\deq  \sqrt{\frac{\|\im M\| + \phi+\eta}{N \eta}}\,, \quad \mbox{and} \quad \wt \zeta \deq (1+\phi)^3 \zeta \,.
\end{equation}
Since $N^{-1}\leq \phi\leq N^{\tau/10}$, $N^{-1+\tau} \le \eta \le \tau^{-1} $, and $\norm{M} = O(1)$ by the assumption of Theorem \ref{thm:main_general}, we have
\begin{equation} \label{eq:bound_on_psi}
\zeta\leq  \wt \zeta\leq (1+\phi)^3  O\Bigg(\sqrt{\frac{1+\phi}{N\eta}+\frac{1}{N}}\,\Bigg) \leq (1+\phi)^3 O\Bigg(\sqrt{\frac{1+\phi}{N\eta}}\,\Bigg) \leq O\big(N^{-3\tau/20}\big)\,,
\end{equation}
and in particular $\wt \zeta \leq 1$. In addition, it follows from the definition of $ \zeta$ that $\zeta \geq N^{-1/2}$.

\paragraph{Using the Ward identity}

The proof of Theorem \ref{thm:main_general} makes constant use of the Ward identity (Lemma \ref{lem:Ward_identity}) and its consequence \eqref{ward_estimate}, usually combined with the Cauchy-Schwarz inequality. For future reference, we collect here several representative examples of such uses, as they appear in the proof. To streamline the presentation, in the actual proof of Theorem \ref{thm:main_general} below, we do not give the details of the applications of the Ward identity. Instead, we tacitly use estimates of the following type whenever we mention the use of the Ward identity. In each of the examples, $\f v,\f w\in \mathbb{S}$ are deterministic and $B$ is a deteministic $N\times N$ matrix satisfying $\norm{B} = O(1)$.

\begin{enumerate}
	\item By \eqref{ward_estimate} and the estimate $G - M = O_\prec(\phi)$ we have
	\begin{equation}\label{eq:Ward_example_1}
	\frac{1}{N}\sum_i |(GB)_{\f v i}|^2 \leq O(1) \frac{\im G_{\f v\f v}}{N\eta} \prec \frac{\im M_{\f v\f v}+\phi}{N\eta} \leq \frac{\norm{\im M} +\phi}{N\eta} \leq \zeta^2\,.
	\end{equation}
	Similarly, 
	\begin{equation}\label{eq:Ward_example_2}
	\frac{1}{N}\sum_i |(GB)_{\f v i}| \leq \bigg(\frac{1}{N}\sum_i |(GB)_{\f v i}|^2\bigg)^{1/2} \prec \zeta\,.
	\end{equation}
	\item Using that $s_{ij} = O(1)$  and \eqref{eq:Ward_example_2} we find
	\begin{multline}\label{eq:Ward_example_3}
	\frac{1}{N^2}\sum_{i,j} \big|G_{\f v \f e_i^j}(GB)_{ji}G_{j\f w}\big| \leq \frac{O(1)}{N} \sum_i|G_{\f v i}|\bigg(\frac{1}{N}\sum_j|(GB)_{ji}G_{j\f w}|\bigg) \\
	\leq \frac{O(1)}{N}\sum_i |G_{\f v i}|\bigg(\frac{1}{N}\sum_j |(GB)_{ji}|^2\bigg)^{1/2}\bigg(\frac{1}{N}\sum_j |G_{j\f w}|^2\bigg)^{1/2}\prec \frac{1}{N}\sum_i |G_{\f v i}|\zeta^2 \prec \zeta^3\,.
	\end{multline}
	\item We have
	\begin{multline}\label{eq:Ward_example_4}
	\frac{1}{N^4}\sum_{i,j,a,b}\big|(GB)_{\f v \f e_i^j}G_{j\f w}G_{\f v a}G_{a\f w} G_{\f e_b^ai}(GB)_{jb}\big| \\
	= \frac{O(1)}{N^3} \sum_{i,j,a}\big|(GB)_{\f v i}G_{j\f w}G_{\f v a}G_{a\f w}\big|\bigg(\frac{1}{N}\sum_b |G_{b i}(GB)_{jb}|\bigg) \prec\frac{1}{N^3} \sum_{i,j,a}\big|(GB)_{\f v i}G_{j\f w}G_{\f v a}G_{a\f w}\big|\zeta^2  \\
	\leq \bigg(\frac{1}{N}\sum_i |(GB)_{\f v i}|\bigg)\bigg(\frac{1}{N}\sum_j |G_{j\f w}|\bigg)\bigg(\frac{1}{N}\sum_a |G_{\f v a}G_{a\f w}|\bigg)\zeta^2 \prec \zeta^6\,.
	\end{multline}
	\item Let $Q$ be a random matrix satisfying $Q=O_\prec(\lambda)$. Then
	\begin{equation}\label{eq:Ward_example_5}
	\frac{1}{N^2} \sum_{i,j} |s_{ij} (GB)_{ji}Q_{ij}| \leq \frac{O(1)}{N^2}\sum_{i,j} |(GB)_{ji}Q_{ij}| \prec \frac{\lambda}{N^2}\sum_{i,j} |(GB)_{ji}| \leq \lambda \zeta\,.
	\end{equation}
\end{enumerate}

\paragraph{Reduction: removal of $\cal J$}
In order to simplify the estimation of $\Pi(G)$ and $\underline{B\Pi(G)}$ we subtract from $\Pi(G)$ a term, denoted below by $\cal J$, which can easily be shown to be sufficiently small. To this end, we split
\begin{equation} \label{eg}
(\cal S(G)G)_{\f v\f w} = \mathcal J_{\f v \f w} + \mathcal K_{\f v \f w} \,, \qquad \mathcal J_{\f v \f w} \deq \frac{1}{N}\sum_{j} G_{j {\f v}^{j}}G_{j\f w}\,, \qquad \mathcal K_{\f v \f w} \deq \frac{1}{N}\sum_{j} G_{jj} G_{ \f v^{j}\f w}\,,
\end{equation}
and denote 
\begin{equation} \label{def_D}
\cal D_{\f v\f w} \deq \Pi(G)_{\f v\f w} - \mathcal J_{\f v\f w} = (I+zG-AG+ \cal K)_{\f v\f w}=(HG +\cal K)_{\f v\f w}\,,
\end{equation}
where we used the identity $I+zG-AG = HG$.

Due to \eqref{eq:bound_on_G} and the Ward identity, we know that 
\begin{equation} \label{egWard}
|\mathcal J_{\f v\f w}| \prec (1+\phi)\cdot \frac{1}{N}\sum_{j}|G_{j \f w}| \prec (1 + \phi) \zeta\,.
\end{equation}
Similarly, using the Ward identity we get
\begin{equation}\label{egWard2}
|\underline{\cal J B}| = \frac{1}{N^2} \bigg|\sum_{i,j} (GB)_{ji}G_{ji}s_{ij}\bigg| \prec \zeta^2\,.
\end{equation}
Because of \eqref{egWard} and \eqref{egWard2}, for the proof of Theorem \ref{thm:main_general} it suffices to prove the following two results.

\begin{proposition} \label{prop:est1}
	Under the assumptions of Theorem \ref{thm:main_general} we have $\cal D = O_\prec\pb{\wt \zeta}$.
\end{proposition}

\begin{proposition} \label{prop:est2}
	Under the assumptions of Theorem \ref{thm:main_general} we have $\ul{B \cal D} = O_\prec\pb{{\wt \zeta}^2}$ provided that $\norm{B} = O(1)$.
\end{proposition}

In both proofs we expand the term $HG$ on the right-hand side of \eqref{def_D} using the cumulant expansion
(Lemma \ref{lem:cumulant_expansion}). Since $H$ is symmetric, for any differentiable function $f=f(H)$ we set
\begin{equation} \label{diff}
\partial_{ij}f(H)\deq\frac{\partial }{\partial H_{ij}}f(H)=\frac{\partial }{\partial H_{ji}}f(H)\deq \frac{\mathrm{d}}{\mathrm{d}t}\Big{|}_{t=0} f\pb{H+t\,\Delta^{ij}}\,,
\end{equation}
where $\Delta^{ij}$ denotes the matrix whose entries are zero everywhere except at the sites $(i,j)$ and $(j,i)$ where they are one: $\Delta^{ij}_{kl} =(\delta_{ik}\delta_{jl}+\delta_{jk}\delta_{il})(1+\delta_{ij})^{-1}$. By differentiation of $G=(H + A -z)^{-1}$ we get
\begin{equation} \label{3.15}
\partial_{ij} G_{\f x\f y}=-(1+\delta_{ij})^{-1}(G_{\f x i}G_{j\f y}+G_{\f x j}G_{i\f y})\,.
\end{equation}
Combining \eqref{3.15}, the assumption $\norm{B}=O(1)$, and \eqref{eq:bound_on_G}, a simple induction gives
\begin{equation}\label{eq:bounds_on_derivatives_of_G}
|\partial^m_{ij}G_{\f x\f y}| \prec (1+\phi)^{m+1},\qquad |\partial^m_{ij}(GB)_{\f x\f y}|\prec (1+\phi)^{m+1}
\end{equation}
for all $\f x,\f y\in \mathbb{X}$ and $m\geq 0$.

\subsection{Proof of Proposition \ref{prop:est1}}

Before starting the proof, we record a preliminary estimate on the derivatives of $\cal D$.

\begin{lemma} \label{3.1}
	Suppose that the assumptions of Theorem \ref{thm:main_general} hold. Then for any fixed $l \ge 0$ we have
	\begin{equation*}
	\partial_{ij}^l \cal D =O_{\prec}(N^{1/2}(1+\phi)^{l+2}\zeta )\,.
	\end{equation*}
\end{lemma} 
\begin{proof}
	We proceed by induction. Fix $\f v, \f w \in \bb S$. Using \eqref{eq:bound_on_G}, we obtain
	\begin{equation*}
	|\cal D_{\f v\f w }| \leq |(HG)_{\f v \f w}|+\frac{1}{N} \sum\limits_j |G_{jj}G_{\f v^j \f w}| \prec |(HG)_{\f v \f w}|+ (1+\phi)^2\,.
	\end{equation*}
	Using the bound \eqref{eq:bound_on_G}, the Ward identity, and Lemma \ref{lem:H_bound}, we therefore get
	\begin{equation*}
	|(HG)_{\f v \f w}| \le \norm{H}  \abs{G\f w} \prec \abs{G\f w} = \sqrt{\frac{\im G_{\f w \f w}}{\eta}} \prec N^{1/2}\zeta\,.
	\end{equation*}
	Since $\zeta \ge N^{-1/2}$ we conclude that
	\begin{equation*}
	|\cal D_{\f v\f w }| \prec N^{1/2}(1+\phi)^2\zeta\,.
	\end{equation*}
	This completes the proof for the case $l = 0$.

	Next, let $l \geq 1$ and suppose that $\partial^m_{ij}\cal D=O_{\prec}(N^{1/2}(1+\phi)^{m+2}\zeta )$ for $0 \leq m \leq l-1$. Using \eqref{3.15} we find
	\begin{equation} \label{diffD}
	(1+\delta_{ij})\partial_{ij} \cal D_{\f v \f w}=-(\cal D_{\f v i}G_{j \f w}+\cal D_{\f v j}G_{i \f w})+U^{ij}_{\f v \f w}\,,
	\end{equation}
	where 
	\begin{equation} \label{3.23}
	U^{ij}_{\f v \f w} \deq v_i G_{j \f w}+v_j G_{i \f w}-\frac{1}{N}\sum\limits_{r}G_{ri}G_{jr}G_{\f v^r \f w}-\frac{1}{N}\sum\limits_{r}G_{rj}G_{ir}G_{\f v^r \f w}\,.
	\end{equation}
	Thus
	\begin{equation} \label{3.24}
	(1+\delta_{ij})\partial_{ij}^l \cal D_{\f v \f w}=-\partial^{l-1}_{ij}(\cal D_{\f v i}G_{j \f w}+\cal D_{\f v j}G_{i \f w})+\partial^{l-1}_{ij}U^{ij}_{\f v \f w}\,.
	\end{equation} 
	We deal with each of the terms on the right-hand side separately. The term $\partial^{l-1}_{ij}(\cal D_{\f v i}G_{j \f w})$
	is estimated using \eqref{eq:bounds_on_derivatives_of_G} as
	\begin{multline*}
	\absb{\partial^{l-1}_{ij}(\cal D_{\f v i}G_{j \f w})} \leq \sum_{m=0}^{l-1} \binom{l-1}{m} \absb{ (\partial^{l-1-m}_{ij}\cal D_{\f v i})(\partial_{ij}^m G_{ j \f w})} \prec \sum_{m=0}^{l-1}(1+\phi)^{m+1} \abs{\partial^{l-1-m}_{ij}\cal D_{\f v i}}\\
	\prec \sum_{m=0}^{l-1}(1+\phi)^{m+1} N^{1/2}(1+\phi)^{l-m+1}\zeta = O(N^{1/2}(1+\phi)^{l+2}\zeta)\,,
	\end{multline*}
	where in the third step we used the induction assumption. The second term of \eqref{3.24} is estimated analogously. Finally, by \eqref{3.15}, \eqref{eq:bounds_on_derivatives_of_G}, \eqref{3.23}, and the bound $|v_i| \le 1$, we can estimate the last term of \eqref{3.24} as
	\begin{equation*}
	\partial^{l-1}_{ij}U^{ij}_{\f v \f w} \prec (1+\phi)^{l-1+3} =(1+\phi)^{l+2} \le N^{1/2}(1+\phi)^{l+2} \zeta\,.
	\end{equation*}
	This concludes the proof.
\end{proof}

Now we turn to the proof of Proposition \ref{prop:est1}. Fix a constant $p\in\N$. Then
\begin{equation}\label{eq:proof_of_sc_equation_1}
\E [|\cal D_{\f v\f w}|^{2p}]	
=\E[(\cal K+HG)_{\f v\f w}\cal D_{\f v\f w}^{p-1}\overline{\cal D}_{\f v\f w}^p]
=\E[\cal K_{\f v\f w}\cal D_{\f v\f w}^{p-1}\overline{\cal D}_{\f v\f w}^p]
+\sum_{i,j}{v}_i\E[H_{ij}G_{j\f w}\cal D_{\f v\f w}^{p-1}\overline{\cal D}_{\f v\f w}^p]\,.
\end{equation}

Using the cumulant expansion (Lemma \ref{lem:cumulant_expansion}) for the second term in \eqref{eq:proof_of_sc_equation_1} with $h=H_{ij}$ and
\begin{equation} \label{f}
f=f_{ij}(H)=G_{j\f w}\cal D_{\f v\f w}^{p-1}\overline{\cal D}_{\f v\f w}^p\,,
\end{equation} we obtain
\begin{equation}\label{eq:proof_of_sc_equation_2}
\E|\cal D_{\f v\f w}|^{2p}
=\E[\cal K_{\f v\f w}\cal D_{\f v\f w}^{p-1}\overline{\cal D}_{\f v\f w}^p]
+\sum_{k=1}^{\ell}X_k+\sum_{i,j}{v}_i  \cal R_{\ell+1}^{ij}\,,
\end{equation}
where we used the notation 
\begin{equation} \label{X_k}
X_k\deq\sum_{i,j}{v}_i\bigg(\frac{1}{k!}\cal C_{k+1}(H_{ij}) \, \E\qb{\partial_{ij}^{k}\big(G_{j\f w}\cal D_{\f v \f w}^{p-1}\overline{\cal D}_{\f v \f w}^p\big)}\bigg)\,.
\end{equation}
Note that the sum in \eqref{eq:proof_of_sc_equation_2} begins from $k = 1$ because $\cal C_1(H_{ij}) = 0$.
Here $\ell$ is a fixed positive integer to be chosen later, and $\cal R_{\ell+1}^{ij}$ is a remainder term defined analogously to $\cal R_{\ell+1}$ in (\ref{R_l+1}). More precisely, for any $t>0$ we have the bound
\begin{equation} \label{tau1}
\begin{aligned}	
\cal R_{\ell+1}^{ij}&=O(1) \cdot \bigg(\E  \sup_{|x| \le |H_{ij}|}\big| \partial_{ij}^{\ell + 1}f({H}^{ij}+x\Delta^{ij})\big|^2\cdot\E\big| {H^{2\ell+4}_{ij}}\mathbf{1}_{|H_{ji}|>t}\big|\bigg)^{1/2}\\&\ \ +O(1) \cdot \E |H_{ij}|^{\ell+2} \cdot  \E \qbb{ \sup_{|x| \le t}\big| \partial_{ij}^{\ell + 1}f({H}^{ij}+x\Delta^{ij})\big|}\,,
\end{aligned}
\end{equation}
where we defined ${H}^{ij}\deq H- H_{ij}\Delta^{ij}$, so that the matrix ${H}^{ij}$ has zero entries at the positions $(i,j)$ and $(j,i)$.
The proof of Proposition \ref{prop:est1} is broken down into the following lemma.

\begin{lemma}\label{prop:main_estimations}
Suppose that $\cal D=O_{\prec}(\lambda)$ for some deterministic $\lambda \ge \wt \zeta$.	Fix $p \ge 2$. Under the assumptions of Theorem \ref{thm:main_general}, we have the following estimates.
	\begin{enumerate}
		\item $\E[\cal K_{\f v\f w}\cal D_{\f v\f w}^{p-1}\overline{\cal D}_{\f v\f w}^p]+X_1= O_\prec (\wt \zeta) \cdot\E|\cal D_{\f v\f w}|^{2p-1} +O_\prec(\wt \zeta \lambda)\cdot\E|\cal D_{\f v\f w}|^{2p-2}$.
		\item For $k \geq 2$, $X_k= \sum_{n=1}^{2p}  O_\prec((\wt \zeta \lambda^2)^{n/3})\cdot\E|\cal D_{\f v \f w}|^{2p-n}$.
		\item For any $D>0$ there exists $\ell \equiv \ell(D) \geq 1$  such that $\cal R_{\ell+1}^{ij}=O(N^{-D})$ uniformly for all $i,j\in \qq{N}$.
	\end{enumerate}
\end{lemma}

Indeed, combining Lemma \ref{prop:main_estimations} for $\ell \equiv \ell(2p+2)$ together with \eqref{eq:proof_of_sc_equation_2} we obtain 
\begin{align*}
\E |\cal D_{\f v\f w}|^{2p} &= 
\sum_{n=1}^{2p}  O_{\prec}((\wt \zeta \lambda^2)^{n/3})\cdot\E|\cal D_{\f v \f w}|^{2p-n}
+O(N^{-2p})
\\
&\leq
\sum_{n=1}^{2p}  O_{\prec}((\wt \zeta \lambda^2)^{n/3})\cdot\big(\E|\cal D_{\f v \f w}|^{2p}\big)^{(2p-n)/2p}
+O(N^{-2p})\,,
\end{align*}
 where in the second step we used H\"older's inequality. Since $\lambda \geq \wt \zeta \geq N^{-1/2}$, we conclude that $N^{-2p} \leq (\wt \zeta \lambda^2)^{2p/3}$, and therefore
$\E |\cal D_{\f v\f w}|^{2p}=O_{\prec}((\wt \zeta \lambda^2)^{2p/3})$. Since $p$ was arbitrary, we conclude from Markov's inequality the implication
\begin{equation} \label{D_self_improv}
\cal D=O_{\prec}(\lambda) \qquad \Longrightarrow \qquad \cal D=O_{\prec}((\wt \zeta \lambda^2)^{1/3})\,,
\end{equation}
for any deterministic $\lambda \geq \wt \zeta$. Moreover, by Lemma \ref{3.1} we have the a priori bound $\cal D=O_{\prec}((1+\phi)^2 N^{1/2}\zeta)$. Hence, an iteration of \eqref{D_self_improv}, analogous to Lemma \ref{lem:self-improv}, yields $\cal D = O_\prec(\wt \zeta)$. This concludes the proof of Proposition \ref{prop:est1}.

What remains, therefore, is the proof of Lemma \ref{prop:main_estimations}.

\begin{proof}[Proof of Lemma \ref{prop:main_estimations} (i)]
	Note that $\cal C_2(H_{ij})=\E [H^2_{ij}]$ since $\E H_{ij} = 0$. Hence,
	\begin{equation*}
	\begin{aligned}
	X_1 &= \sum_{i,j}{v}_{i}\cal C_2(H_{ij})\E\big[\partial_{ij}\big(G_{j\f w}\cal D_{\f v\f w}^{p-1}\overline{\cal D}_{\f v\f w}^p\big)\big]\\
	& =-\sum_{i,j}{v}_i\E[H_{ij}^2]\cdot\E\big[(G_{ji}G_{j\f w}+G_{jj}G_{i\f w})(1+\delta_{ij})^{-1}\cal D_{\f v\f w}^{p-1}\overline{\cal D}_{\f v\f w}^p\big]
	+\sum_{i,j}v_i\E\big[H_{ij}^2\big]\E\big[G_{j\f w}\partial_{ij}\big(\cal D_{\f v\f w}^{p-1}\overline{\cal D}_{\f v\f w}^p\big)\big]\\
	& = -\mathbb{E}\big[(\mathcal{S}(G)G))_{\mathbf{v}\mathbf{w}}\mathcal{D}_{\mathbf{v}\mathbf{w}}^{p-1}\overline{\mathcal{D}}_{\mathbf{v}\mathbf{w}}^{p}\big]+ \sum_{i,j}{v}_i\E\big[H_{ij}^2\big]\E\big[G_{j\f w}\partial_{ij}\big(\cal D_{\f v\f w}^{p-1}\overline{\cal D}_{\f v\f w}^p\big)\big]\,,
	\end{aligned}
	\end{equation*}	
	and consequently, recalling \eqref{eg}, we have
	\begin{equation}\label{eq:estimating_X_1_1}
	\E[\cal K_{\f v\f w}\cal D_{\f v\f w}^{p-1}\overline{\cal D}_{\f v\f w}^p]+X_1
	= -\E[\cal J_{\f v\f w}\cal D_{\f v\f w}^{p-1}\overline{\cal D}_{\f v\f w}^p]+\sum_{i,j}{v}_i\E[H_{ij}^2]\E[G_{j\f w}\partial_{ij}\big(\cal D_{\f v\f w}^{p-1}\overline{\cal D}_{\f v\f w}^p\big)]\,.
	\end{equation}
	Since \eqref{egWard} implies $\E[\cal J_{\f v\f w}\cal D_{\f v\f w}^{p-1}\overline{\cal D}_{\f v\f w}^p] = O_\prec( \wt \zeta ) \cdot\E|\cal D_{\f v\f w}|^{2p-1}$, it remains to estimate the second term on the right-hand side of \eqref{eq:estimating_X_1_1}.
	
    By \eqref{diffD} and \eqref{3.23} we have	
	\begin{equation} \label{viHij}
	\begin{aligned}
	&\sum_{i,j}{v}_i \E[H_{ij}^2] \, \E\big[G_{j\f w}\partial_{ij}\big(\cal D_{\f v\f w}^{p-1}\overline{\cal D}_{\f v\f w}^p\big)\big] \\
	=&\, (p-1)\sum_{i,j}{v}_i \E[H_{ij}^2] \, \E\big[G_{j\f w} \cal D_{\f v\f w}^{p-2}\overline{\cal D}_{\f v\f w}^{p}\partial_{ij}\cal D_{\f v\f w}\big] +p\sum_{i,j}{v}_i\E[H_{ij}^2] \, \E\big[G_{j\f w}|\cal D_{\f v\f w}|^{2p-2} \partial_{ij}\overline{\cal D}_{\f v\f w}\big]\\
	=&\,\frac{p-1}{N}\sum_{i,j}   \E\bigg[v_is_{ij} G_{j \f w}\cal D_{\f v\f w}^{p-2}\overline{\cal D}_{\f v\f w}^p\bigg(-\cal D_{\f v i}G_{j \f w}-\cal D_{\f v j}G_{i \f w}+v_i G_{j \f w}+v_j G_{i \f w}\\
	&-\frac{1}{N}\sum\limits_{r} G_{ri}G_{jr}G_{\f v^r \f w}-\frac{1}{N}\sum\limits_{r}G_{rj}G_{ir}G_{\f v^r \f w}\bigg)\bigg]+\frac{p}{N}\sum_{i,j} \E\bigg[v_is_{ij} G_{j\f w}|\cal D_{\f v\f w}|^{2p-2}\\
	&\bigg(-\ol{\cal D}_{\f v i}\ol G_{j \f w}-\ol{\cal D}_{\f v j}\ol G_{i \f w}+v_i \ol G_{j \f w}+v_j \ol G_{i \f w}-\frac{1}{N}\sum\limits_{r}\ol G_{ri}\ol G_{jr}\ol G_{\f v^r \f w}-\frac{1}{N}\sum\limits_{r}\ol G_{rj}\ol G_{ir}\ol G_{\f v^r \f w}\bigg)\bigg]\,.\\
	\end{aligned}
	\end{equation}
	Next, we show that each of the terms on the right-hand side of \eqref{viHij} is bounded in absolute value by  $O_\prec(\wt \zeta \lambda) \cdot\E|\cal D_{\f v\f w}|^{2p-2}$. 
We only discuss the first, third, and fifth terms on the right-hand side of \eqref{viHij}; all the others are estimated analogously. 
	
	We start with the first term on the right-hand side of \eqref{viHij}. By the assumption $\cal D=O_{\prec}(\lambda)$ we know
	\begin{equation*}
	\sup_{j \in \qq{N}}|\cal D_{\f v \f v ^j}| \prec \lambda\,.
	\end{equation*}
Thus, with the help of the estimate \eqref{eq:Ward_example_1}, we obtain 
	\begin{equation} \label{wardexample}
	\begin{aligned}
	& \bigg|\frac{p-1}{N}\sum_{i,j} \E\Big[-v_is_{ij}G_{j \f w}\cal D_{\f v\f w}^{p-2}\overline{\cal D}_{\f v\f w}^p\cal D_{\f v i}G_{j \f w}\Big]\bigg|=\bigg|\frac{p-1}{N}\sum_{j}\E\Big[G_{j \f w}\cal D_{\f v\f w}^{p-2}\overline{\cal D}_{\f v\f w}^p\cal D_{\f v \f v^{j}}G_{j \f w}\Big]\bigg|\\
	&\ = O_\prec(\lambda)\cdot\E\bigg[\frac{1}{N}\sum_j|G_{j\f w}|^2 \cdot \big|\cal D_{\f v\f w}\big|^{2p-2}\bigg]
	=O_\prec(\lambda\zeta^2)\cdot\E\big|\cal D_{\f v\f w}\big|^{2p-2}\,,
	\end{aligned}
	\end{equation}
	and the desired estimate follows from $\zeta \le \wt \zeta \le 1$.
	
	Next, we estimate the third term on the right-hand side of \eqref{viHij}. Using the Ward identity we get
	\begin{equation*}
	\begin{aligned}
	&\bigg|\frac{p-1}{N}\sum_{i,j}   \E\Big[-v_is_{ij}G_{j \f w}\cal D_{\f v\f w}^{p-2}\overline{\cal D}_{\f v\f w}^pv_iG_{j \f w}\Big]\bigg|\\
	&\ =O(1) \cdot\sum\limits_i |v_i|^2 \cdot\E \bigg[\frac{1}{N}\sum\limits_j|G_{j \f w}|^2\cdot \big|\cal D_{\f v\f w}\big|^{2p-2} \bigg]=O_{\prec}(\zeta^2)\cdot \E \big|\cal D_{\f v\f w}\big|^{2p-2}\,,
	\end{aligned}
	\end{equation*}
	and the result follows from the inequality $\zeta \leq \wt \zeta \leq \lambda$.
	
	Finally, we estimate the fifth term on the right-hand side of \eqref{viHij}. By the Cauchy-Schwarz inequality and the Ward identity we have
	\begin{equation*}
	\begin{aligned}
	&\bigg|\frac{p-1}{N^2}\sum_{i,j,r}   \E\Big[-v_is_{ij}G_{j \f w}\cal D_{\f v\f w}^{p-2}\overline{\cal D}_{\f v\f w}^pG_{ri}G_{jr}G_{\f v^{r} \f w}\Big]\bigg|=\bigg|\frac{p-1}{N^2}\sum_{j,r}\E\Big[G_{j \f w}\cal D_{\f v\f w}^{p-2}\overline{\cal D}_{\f v\f w}^pG_{r\f v^j}G_{jr}G_{\f v^{r} \f w}\Big]\bigg|\\
	&\ =O_{\prec}((1+\phi)^2)\cdot \frac{1}{N}\sum\limits_r \E \bigg[\frac{1}{N}\sum\limits_{j}|G_{j \f w}G_{j r}|\cdot \big|\cal D_{\f v \f w}\big|^{2p-2}\bigg]=O_{\prec}((1+\phi)^2\zeta^2)\cdot \E \big|\cal D_{\f v\f w}\big|^{2p-2}\,,
	\end{aligned}
	\end{equation*}
	and the desired estimate follows from $(1+\phi)\zeta \leq \wt \zeta \leq \lambda$. One can check that the estimates of other terms on the right-hand side of \eqref{viHij} follow analogously. We conclude that the term on the right-hand side of \eqref{eq:estimating_X_1_1} is bounded by
	\begin{equation*}
	O_\prec(\wt \zeta \lambda) \cdot\E|\cal D_{\f v\f w}|^{2p-2}\,,
	\end{equation*}
as claimed. 
\end{proof}

\begin{proof} [Proof of Lemma \ref{prop:main_estimations} (ii)]
	Fix $k \ge 2$. By Lemma \ref{lem:cumulant_factos_estimate} we have
	\begin{equation*}
	\begin{aligned}
	X_k &=O\big(N^{-\frac{k+1}{2}}\big)\cdot \sum_{i,j} \E\big|v_i\partial_{ij}^{k}\big(G_{j\f w}\cal D_{\f v \f w}^{p-1}\overline{\cal D}_{\f v \f w}^p\big)\big|\\&= O(N^{-\frac{k+1}{2}})\cdot \sum_{\substack{r,s,t\geq 0 \\ r+s+t=k}} \sum_{i,j}\E\big|v_i(\partial_{ij}^rG_{j \f w})(\partial_{ij}^s \cal D_{\f v\f w}^{p-1})(\partial_{ij}^t \overline{\cal D}_{\f v\f w}^p)\big|\,.
	\end{aligned}
	\end{equation*}
	As the sum over $r,s,t$ is finite it suffices to deal with each term separately. To simplify notation, we drop the complex conjugates of $Q$ (which play no role in the subsequent analysis), and estimate the quantity

	\begin{equation}\label{eq:whatever_kk}
	O(N^{-\frac{k+1}{2}}) \cdot \sum_{i,j}\E\big|v_i\big(\partial_{ij}^rG_{j \f w}\big)\big(\partial_{ij}^{k-r} \cal D_{\f v\f w}^{2p-1})\big|
	\end{equation}
	for $r = 0, \dots, k$.
	Computing the derivative $\partial_{ij}^{k - r}$, we find that \eqref{eq:whatever_kk} is bounded by a sum of terms of the form
	\begin{equation} \label{sqll}
	O(N^{-\frac{k+1}{2}}) \cdot \sum_{i,j}\E\,\Big|v_i\big(\partial_{ij}^rG_{j \f w
		}\big) \Big(\prod_{m=1}^{q}\big(\partial_{ij}^{l_m}\cal D_{\f v\f w}\big)\Big)\cal D_{\f v\f w}^{2p-1-q}\Big|\,,
	\end{equation}
	where the sum ranges over integers $q = 0, \dots, (k - r) \wedge (2p - 1)$ and $l_1, \dots, l_q \geq 1$ satisfying $l_1 + \cdots + l_q = k - r$.
	Using Lemma \ref{3.1}, we find that \eqref{sqll} is bounded by
	\begin{equation*}
	O_\prec(N^{-\frac{k+1}{2}}) \cdot \sum_{i,j}\E\Big[\big|v_i(\partial_{ij}^rG_{j \f w})\big| N^{\frac{q}{2}}(1+\phi)^{k-r+2q}\zeta^{q} |
	\cal D_{\f v\f w}|^{2p-1-q}\Big]\,.
	\end{equation*}	
	Note that by \eqref{3.15} the derivative $\partial_{ij}^rG_{j \f w}$ can be written as a sum of terms, each of which is a product of $r+1$ entries of the matrices $G$, with one entry of the form $G_{i \f w}$ or $G_{j \f w}$. For definiteness, we suppose that this entry is $G_{i \f w}$ in the product. Using \eqref{eq:bound_on_G}, the Cauchy-Schwarz inequality, and the Ward identity, we get
	\begin{equation*}
	\sum\limits_{i,j}\big|v_i\partial_{ij}^rG_{i \f w}\big| \prec (1+\phi)^{r}\cdot \sum\limits_{i,j}\big|v_iG_{i \f w}\big|
	\prec N^{\frac{3}{2}}(1+\phi)^r\zeta\,,
	\end{equation*}
	and therefore
	\begin{equation*} 
	\begin{aligned} 
	\eqref{sqll} &\prec N^{\frac{q+2-k}{2}} \cdot (1+\phi)^{k+2q}\zeta^{q+1}\E|\cal D_{\f v\f w}|^{2p-1-q}
	\\
	& = \big(N^{-\frac{1}{2}}(1+\phi)\big)^{k-2-q} \cdot (1+\phi)^{3q+2}\zeta^{q+1} \E|\cal D_{\f v\f w}|^{2p-1-q}
	\\ 
	&\leq\big(N^{-\frac{1}{2}}(1+\phi)\big)^{k-2-q} \cdot \wt \zeta^{q+1} \E|\cal D_{\f v\f w}|^{2p-1-q}\,,
	\end{aligned}
	\end{equation*}
	where for the last inequality we used $(1+\phi)^3\zeta= \wt \zeta$. Clearly, if $q\leq k-2$, we conclude that \eqref{sqll} is bounded by $\wt \zeta^{q+1} \E[|\cal D_{\f v\f w}|^{2p-1-q}]$, and hence the desired estimate follows from the assumption $\wt \zeta \le \lambda$.
	
	What remains, therefore, is to estimate \eqref{sqll}  for $q \geq k - 1$, which we assume from now on.
	Because $k \geq 2$ by assumption, we find that $q \geq 1$. Moreover, since $q \le k-r$, we find that $r \leq 1$. Thus, it remains to consider the three cases  $(r,q)=(0,k)$, $(r,q)=(1,k-1)$ and $(r,q)=(0,k-1)$. We deal with them separately.
	\paragraph{The case $(r,q)=(0,k)$} In this case $l_1=l_2=\dots=l_q=1$, so that \eqref{sqll} reads
	\begin{multline}\label{eq:whatever_31}
	O(N^{-\frac{k+1}{2}}) \cdot \sum_{i,j} \E\,\Big|v_iG_{j \f w} (\partial_{ij}\cal D_{\f v\f w})^k \cal D_{\f v\f w}^{2p-1-k}\Big|\\
	\prec
	N^{-\frac{k+1}{2}}\cdot(N^{\frac{1}{2}}(1+\phi)^3\zeta)^{k-2} \cdot \sum_{i,j} \E\,\Big|v_iG_{j \f w} (\partial_{ij}\cal D_{\f v\f w})^2 \cal D_{\f v\f w}^{2p-1-k}\Big|\\
	=
	\wt \zeta^{k-2}\cdot N^{-\frac{3}{2}}\cdot\sum\limits_{i,j}\E\,\Big|v_iG_{j \f w} (\partial_{ij}\cal D_{\f v\f w})^2 \cal D_{\f v\f w}^{2p-1-k}\Big|\,,
	\end{multline}
	where we used Lemma \ref{3.1} and the assumption $k \ge 2$.	
	By \eqref{diffD}, \eqref{eq:bound_on_G}, the Ward identity, the bound $\lambda \geq \zeta$, and the assumption $D=O_{\prec}(\lambda)$ we have
	\begin{equation} \label{3.46}
	(1+\delta_{ij})\partial_{ij}\cal D_{\f v\f w}
	=\ O_{\prec}((1+\phi)\lambda)+v_iG_{j \f w}+v_jG_{i \f w}\,.
	\end{equation}
	Thus,
	\begin{equation} \label{3.47}
	\sum\limits_{i,j}\big|v_iG_{j \f w}(\partial_{ij}\cal D_{\f v\f w})^2\big|\leq\sum\limits_{i,j} \Big| v_i G_{j \f w} (O_{\prec}((1+\phi)\lambda)+v_iG_{j \f w}+v_jG_{i \f w})^2\Big|\,.
	\end{equation}
	We expand the square on the right-hand side of \eqref{3.47} and estimate the result term by term, using \eqref{eq:bound_on_G}, the Ward identity, and $|\f v|=1$. For example we have
	\begin{equation*}
	\sum_{i,j}\big|v_iG_{j \f w} O_{\prec}((1+\phi)\lambda)^2\big|
	\prec N^{\frac{3}{2}}(1+\phi)^2 \zeta \lambda^2\,,
	\end{equation*}
	and
	\begin{equation*}
	\sum\limits_{i,j}\big|v_iG_{j \f w}v_jG_{i \f w}v_iG_{j \f w} \big| \prec (1+\phi) \cdot \sum\limits_{i} v_i^2 \sum\limits_{j}|G_{j \f w}|^2\prec N(1+\phi)\zeta^2 \leq N^{\frac{3}{2}}\wt \zeta \lambda^2\,,
	\end{equation*}
	where in the last step we used the estimates $(1+\phi)\zeta \le \wt \zeta$ and $N^{-1/2}\le \zeta \le \lambda$. By estimating other terms on the right-hand side of \eqref{3.47} in a similar fashion, we get the bound
	\begin{equation*}
	\sum\limits_{i,j}\big|v_iG_{j \f w}(\partial_{ij}\cal D_{\f v\f w})^2\big|\prec N^{\frac{3}{2}}\wt \zeta \lambda^2\,.
	\end{equation*}
	Together with \eqref{eq:whatever_31} we therefore conclude that
	\begin{equation*}
		O(N^{-\frac{k+1}{2}}) \cdot \sum_{i,j} \E\Big|v_iG_{j \f w} (\partial_{ij}\cal D_{\f v\f w})^k \cal D_{\f v\f w}^{2p-1-k}\Big| \prec \wt \zeta^{k-1}\lambda^2 \E |\cal D_{\f v \f w}|^{2p-k-1}\,.
	\end{equation*}
	The desired estimate then follows from the assumptions $\wt \zeta \le \lambda$ and $k \ge 2$.

	\paragraph{The case $(r,q)=(1,k-1)$} We apply a similar argument to the one from the previous case. In this case $l_1=l_2=\dots=l_{k-1}=1$, and \eqref{sqll} reads
	\begin{multline}\label{eq:whatever_101}
	O(N^{-\frac{k+1}{2}})\cdot  \sum_{i,j}\E\big|v_i(\partial_{ij}G_{j \f w})(\partial_{ij}\cal D_{\f v\f w})^{k-1} \cal D_{\f v\f w}^{2p-k}\big|\\\prec N^{-\frac{k+1}{2}}\cdot(N^{\frac{1}{2}}(1+\phi)^3\zeta)^{k-2} \cdot \sum_{i,j} \E\Big|v_i(\partial_{ij}G_{j \f w}) (\partial_{ij}\cal D_{\f v\f w}) \cal D_{\f v\f w}^{2p-k}\Big|\\
	=\wt \zeta^{k-2}\cdot N^{-\frac{3}{2}}\cdot\sum\limits_{i,j}\E\Big|v_i(\partial_{ij}G_{j \f w}) (\partial_{ij}\cal D_{\f v\f w}) \cal D_{\f v\f w}^{2p-k}\Big|\,,
	\end{multline}
	where in the second step we used Lemma \ref{3.1} and the assumption $k \ge 2$.
	Using \eqref{3.15} and \eqref{3.46} to rewrite $\partial_{ij}G_{j \f w}$ and one factor of $\partial_{ij}\cal D_{\f v \f w}$ respectively, we have
	\begin{equation} \label{3.53}
	\sum_{i,j} \big|v_i(\partial_{ij}G_{j \f w})(\partial_{ij}\cal D_{\f v \f w})\big|\le\sum_{i,j}\big|v_i(G_{ij}G_{j \f w}+G_{ji}G_{j \f w})(O_{\prec}((1+\phi)\lambda)+v_iG_{j \f w}+v_jG_{i \f w})\big|\,.
	\end{equation}
	  We expand the right-hand side of the above and estimate the result term by term, using \eqref{eq:bound_on_G}, the Ward identity, and $|\f v|=1$. For example we have 
	  \begin{equation*}
	  \begin{aligned}
	  \sum_{i,j}\big|v_iG_{ij}G_{j\f w}\cdot v_jG_{i \f w}\big|&\prec (1+\phi)\cdot \bigg(\sum_{i,j}v_i^2\,|G_{j \f w}|^2\bigg)^{\frac{1}{2}}\cdot \bigg(\sum_{i,j}v_j^2\,|G_{i \f w}|^2\bigg)^{\frac{1}{2}} \\
	  &\prec (1+\phi)N\zeta^2 \leq N^{\frac{3}{2}}\wt \zeta^2\,,
	  \end{aligned}
	  \end{equation*}
	  and
	\begin{equation*}
	\sum_{i,j}\big|v_iG_{ij}G_{j \f w} \cdot O_{\prec}((1+\phi)\lambda)\big| \prec
	N^{\frac{3}{2}}(1+\phi)^2\zeta \lambda\,.
	\end{equation*}
    By estimating the other terms on the right-hand side of \eqref{3.53} in a similar fashion, we get the bound
    \begin{equation*}
    \sum_{i,j} \big|v_i(\partial_{ij}G_{j \f w}(\partial_{ij}\cal D_{\f v \f w}))\big| \prec N^{\frac{3}{2}}\wt \zeta\lambda\,.
    \end{equation*}
    Together with \eqref{eq:whatever_101} we conclude that
    \begin{equation*}
    O(N^{-\frac{k+1}{2}})\cdot  \sum_{i,j}\E \big|v_i(\partial_{ij}G_{j \f w})(\partial_{ij}\cal D_{\f v\f w})^{k-1} \cal D_{\f v\f w}^{2p-k}\big| \prec \wt \zeta^{k-1}\lambda\, \E  |\cal D_{\f v \f w}|^{2p-k}.
    \end{equation*}
    The desired estimate then follows from the assumptions $\wt \zeta \le \lambda$ and $k \ge 2$.
	
	\paragraph{The case $(r,q)=(0,k-1)$} Since $l_1+\dots+l_{k-1}=k$ and $l_m\geq 1$ for every $m\in \{1,\dots,k-1\}$, there exists exactly one $m$ such that $l_m=2$ and the remaining $l_m$'s are $1$. Hence, \eqref{sqll} reads
	\begin{multline} \label{porsas1}
	O(N^{-\frac{k+1}{2}}) \cdot \sum_{i,j}\E\big|v_iG_{j \f w}(\partial_{ij}^2 \cal D_{\f v\f w})(\partial_{ij}\cal D_{\f v\f w})^{k-2} \cal D_{\f v\f w}^{2p-k}\big|\\
	\prec \wt \zeta^{k-2}\cdot N^{-\frac{3}{2}}\cdot \sum_{i,j}\E\big|v_iG_{j \f w}(\partial_{ij}^2 \cal D_{\f v\f w})\cal D_{\f v\f w}^{2p-k}\big| \,,
	\end{multline}
	where we used Lemma \ref{3.1}. By \eqref{diffD} we have
	\begin{equation*}
	\sum_{i,j} \big|v_iG_{j \f w}(\partial^2_{ij}\cal D_{\f v \f w})\big| \leq \sum_{i,j} \big|v_iG_{j \f w}\partial_{ij}(-\cal D_{\f v i}G_{j \f w}-\cal D_{\f v j}G_{i \f w}+U^{ij}_{\f v \f w})\big|\,,
	\end{equation*}
	where $U$ is defined as in \eqref{3.23}. Now we apply the differential $\partial_{ij}$ on the right-hand side, and estimate the result term by term. By using the bounds \eqref{eq:bound_on_G} and $\cal D=O_{\prec}(\lambda)$, the Ward identity, and $\abs{\f v}=1$ in a similar fashion as in the previous two cases, we get
	\begin{equation*}
	\sum_{i,j} \big|v_iG_{j \f w}(\partial^2_{ij}\cal D_{\f v \f w})\big|\prec N^{\frac{3}{2}}\wt \zeta\lambda\,.
	\end{equation*}
	Together with \eqref{porsas1} we conclude that
	\begin{equation*}
	O(N^{-\frac{k+1}{2}}) \cdot \sum_{i,j}\E\big|v_iG_{j \f w}(\partial_{ij}^2 \cal D_{\f v\f w})(\partial_{ij}\cal D_{\f v\f w})^{k-2} \cal D_{\f v\f w}^{2p-k}\big|\prec\wt \zeta^{k-1}\lambda\,\E\big|\cal D_{\f v\f w}\big|^{2p-k}\,.
	\end{equation*}
	The desired estimate then follows from the assumptions $\wt \zeta \le \lambda$ and $k \ge 2$. This concludes the proof.
\end{proof}

\begin{proof}[Proof of Lemma \ref{prop:main_estimations} (iii)]
The remainder term of the cumulation expansion was analysed in a slightly different context in \cite{HK}, and we shall follow the same method. 

Let $D>0$ be given. Fix $i,j$, and choose $t\deq N^{\tau/5-1/2}$ in  \eqref{tau1}. Define $S \deq  H_{ij}\Delta^{ij}$, where we recall the notation $\Delta^{ij}_{kl} \deq (\delta_{ik}\delta_{jl}+\delta_{jk}\delta_{il})(1+\delta_{ij})^{-1}$. Then we have $ H^{ij}= H-S$.
Let $\hat{G} \deq  (H^{ij}-E-\mathrm{i}\eta)^{-1}$.
We have the resolvent expansions
\begin{equation} \label{eqn: A.11}
\hat{G}=G+(GS)G+(GS)^2\hat{G}
\end{equation}
and
\begin{equation} \label{eqn: A.12}
G=\hat{G}-(\hat{G}S)\hat{G}+(\hat{G}S)^2G\,.
\end{equation}
Note that only two entries of $S$ are nonzero, and they are stochastically dominated by $N^{-1/2}$. Then the bound
\begin{equation*}
\sup_{\f x ,\f y \in \bb S}|\hat G_{\f x \f y}| = \norm{\hat G} \le \eta^{-1} \le N^{1-\tau}\,,
\end{equation*}
together with \eqref{eq:bound_on_G}, \eqref{eqn: A.11}, and the assumption $\phi \le N^{\tau/10}$, show that
\begin{equation*}
\max_{\f x ,\f y \in \bb X} |\hat{G}_{\f x \f y}| \prec 1+\phi\,.
\end{equation*}
Combining with \eqref{eqn: A.12}, the trivial bound $\sup_{\f x ,\f y \in \bb S}|G_{\f x \f y}| \leq N^{1-\tau}$, and the fact $\hat{G}$ is independent of $S$, we have
\begin{equation} \label{3.67}
\max_{\f x, \f y \in \bb X} \sup_{|H_{ij}|\leq t} |G_{\f x \f y}|\prec 1+\phi\,.
\end{equation}
Then a simple induction using \eqref{3.15} implies
\begin{equation} \label{3.68}
\max_{\f x, \f y \in \bb X} \sup_{|H_{ij}|\leq t} |\partial_{ij}^mG_{\f x \f y}|\prec (1+\phi)^{m+1}
\end{equation}
for every fixed $m \in \N$.
Also, by Assumption \ref{ass:H} we have a trivial bound
\begin{equation} \label{3.69}
\max_{a, b \in \qq{N}}\sup_{|H_{ij}|\le t}|H_{ab}| \prec 1\,.
\end{equation}

Now let us estimate the second term on the right-hand side of \eqref{tau1}. By our definition of $f$ in \eqref{f}, together with \eqref{3.15}, \eqref{diffD}, \eqref{3.68} and \eqref{3.69}, one easily shows that
\begin{equation*}
\sup\limits_{|x|\le t} \Big| \partial_{ji}^{\ell+1}f\big(H^{ij}+x\Delta^{(ij)}\big)\Big| \prec (1+\phi)^{\ell+1} \cdot N^{100p}
\end{equation*}
for any fixed $\ell \in \N$.
Note that $\E |H_{ij}|^{\ell+2}=O(N^{-(\ell+2)/2})$, and by the bound $\phi \le N^{\tau/10}$ we can find $\ell \equiv \ell(D,p)\ge 1$ such that 
\begin{equation*}
\E |H_{ij}|^{\ell+2} \cdot   \E \qbb{\sup_{|x| \le t}\Big| \partial_{ij}^{\ell + 1}f({H}^{ij}+x\Delta^{ij})\Big|} =O(N^{-D})\,.
\end{equation*}

Finally, we estimate the first term on the right-hand side of \eqref{tau1}. For $a,b\in \qq{N}$, we look at $H_{ab}$ as a function of $H$. Note that we have the bound
\begin{equation*}
\sup_{|x|\leq |H_{ij}|}\absb{H_{ab}(H^{ij}+x\Delta^{ij})} \le |H_{ab}|
\end{equation*}
uniformly for all $a,b \in \qq{N}$. Together with Assumption \ref{ass:H}, \eqref{3.15}, \eqref{diffD}, and the trivial bound $\sup_{\f x , \f y \in \bb S}\abs{G_{\f x \f y}} \le N$, we have
\begin{equation} \label{3.73}
\E\sup_{|x| \le |H_{ij}|}\Big| \partial_{ij}^{\ell + 1}f({H}^{ij}+x\Delta^{ij})\Big|^2 =O\big(N^{O_{p,\ell}(1)}\big)\,.
\end{equation}
From Assumption \ref{ass:H} we find $\max_{i,j}|H_{ij}| \prec N^{-1/2}$, thus by Cauchy-Schwarz inequality we have
\begin{equation} \label{3.74}
\E\big| {H^{2\ell+4}_{ij}}\mathbf{1}_{|H_{ji}|>t}\big|=O_{\prec}(N^{-(l+2)}) \cdot \P (|H_{ij}|>t)=O(N^{-2D-O_{p,\ell}(1)})\,.
\end{equation} 
A combination of \eqref{3.73}, and \eqref{3.74} shows that the first term on the right-hand side of \eqref{tau1} is bounded by $O(N^{-D})$.

This completes the proof of $\cal R^{ij}_{\ell+1}=O(N^{-D})$, and our steps also show that the bound is uniform for all $i, j \in \qq{N}$.
\end{proof}

\subsection{Proof of Proposition \ref{prop:est2}}

For a fixed $p\in\N$ write
\begin{equation}\label{eq:proof_of_averaging_1}
\begin{aligned}
\E |\underline{\cal D B}|^{2p}
&=\E[\underline{(\cal K+HG)B}\cdot \underline{\cal D B}^{p-1}\overline{\underline{\cal D B}}^p]\\
&=\E[\underline{\cal K B}\cdot \underline{\cal D B}^{p-1}\underline{\cal D B}^p]
+\frac{1}{N}\sum_{i,j}\E[H_{ij}(GB)_{ji}\underline{\cal D B}^{p-1}\overline{\underline{\cal D B}}^p]\,.
\end{aligned}
\end{equation}
Using the cumulant expansion (Lemma \ref{lem:cumulant_expansion}) for the second term in \eqref{eq:proof_of_averaging_1} with $h=H_{ij}$ and $f=f_{ij}(H)=(GB)_{ji}\underline{\cal D B}^{p-1}\overline{\underline{\cal D B}}^p$, we obtain
\begin{equation}\label{eq:proof_of_averaging_2}
\E|\underline{\cal D B}|^{2p}
=\E[\underline{\cal K B}\cdot \underline{\cal D B}^{p-1}\overline{\underline{\cal D B}}^p]+\sum_{k=1}^{\ell}Y_k+\sum_{i,j}  \tilde {\cal R}_{\ell+1}^{ij}\,,
\end{equation}
where we used the notation 
\begin{equation}\label{eq:Y_k}
Y_k\deq\frac{1}{N}\sum_{i,j}\frac{1}{k!}\cal C_{k+1}(H_{ij})\E\Big[\partial_{ij}^{k}\big((GB)_{ji}\underline{\cal D B}^{p-1}\overline{\underline{ D B}}^p\big)\Big]\,.
\end{equation}
Note that the sum in \eqref{eq:proof_of_averaging_2} begins from $k = 1$ because $\cal C_1(H_{ij}) = 0$. Here $\ell$ is a fixed positive integer to be chosen later, and $\tilde {\cal R}_{\ell+1}^{ij}$ is a remainder term defined analogously to ${\cal R}^{ij}_{\ell+1}$ in \eqref{tau1}.

As in the proof of Proposition \ref{prop:est1}, the proof can be broken down into the following lemma.

\begin{lemma}\label{prop:main_estimations_averaging}
	Fix $p \ge 2$. Under the assumptions of Theorem \ref{thm:main_general} (ii), we have the following estimates.
	\begin{enumerate}
		\item $\E[\underline{\cal K B}\cdot \underline{\cal D B}^{p-1}\overline{\underline{\cal D B}}^p]+Y_1= O_\prec (\zeta^2) \cdot\E|\underline{\cal D B}|^{2p-1} +O_\prec((1+\phi)^{4} \zeta^4)\cdot\E|\underline{\cal D B}|^{2p-2}$.
		\item For $k \geq 2$, $Y_k=(1+\phi)^{6}\sum_{n=1}^{2p}  O_\prec( \zeta^{2n})\cdot\E|\underline{\cal D B}|^{2p-n}$.
		\item For any $D>0$, there exists $\ell\equiv \ell(D)\geq 1$ such that $\cal {\tilde R}_{\ell+1}^{ij}=O(N^{-D})$ uniformly for all $i,j\in \qq{N}$. 
	\end{enumerate}
\end{lemma}

Indeed, combining Lemma \ref{prop:main_estimations_averaging} for $\ell=\ell(4p+2)$ together with \eqref{eq:proof_of_averaging_2} we obtain 
\begin{equation*}
\E |\underline{\cal D B}|^{2p}= 
(1+\phi)^6\sum_{n=1}^{2p}  O_{\prec}(\zeta^{2n})\cdot\E|\underline{\cal D B}|^{2p-n}
+O(N^{-4p})\,.
\end{equation*}
From H\"older's inequality we find
\begin{equation*}
\E |\underline{\cal D B}|^{2p}\leq 
(1+\phi)^6\sum_{n=1}^{2p}  O_{\prec}(\zeta^{2n})\cdot(\E|\underline{\cal D B}|^{2p})^{(2p-n)/2p}	+O(N^{-4p})\,,
\end{equation*}
which by Young's inequality implies $\E |\underline{\cal D B}|^{2p}=O_{\prec}(((1+\phi)^6\zeta^{2})^{2p})$, because $\zeta\geq N^{-1/2}$. Since $p$ was arbitrary, Proposition \ref{prop:est2} follows by Markov's inequality.

What remains, therefore, is the proof of Lemma \ref{prop:main_estimations_averaging}. A crucial ingredient in the proof of Lemma \ref{prop:main_estimations_averaging} is an upper bound on the absolute value of partial derivatives of $\underline{\cal D B}$. To this end, 
define the matrix
\begin{equation}\label{eq:defn_of_Q}
Q_{ij} \deq \frac{1}{N} (GBHG)_{ij} + \frac{1}{N^2}\sum_{a,b} G_{ia}G_{aj}(GB)_{\f e_b^a b}+\frac{1}{N^2}\sum_{a,b} G_{aa}(GB)_{ib}G_{\f e_b^a  j}
\end{equation}
(recall the notation \eqref{eq:tilted_vectors}) and note that
\begin{equation}	\label{eq:relation_of_der_avg_DB_and_Q}
(1+\delta_{ij})\partial_{ij}\underline{\cal D B} = \frac{1}{N}(GB)_{ij}+\frac{1}{N}(GB)_{ji}-Q_{ij}-Q_{ji}\,.
\end{equation}
The following estimate on $Q$ and its derivatives is the key tool behind the proof of Lemma \ref{prop:main_estimations_averaging}.

\begin{lemma}\label{lem:estimation_of_Q} Suppose that the assumptions of Theorem \ref{thm:main_general} (ii) hold.
	\begin{enumerate}
		\item For every fixed $l\geq 0$ we have $\partial_{ij}^l Q = O_\prec((1+\phi)^{l+1}\zeta^2)$.
		\item We have $Q = O_\prec((1+\phi)^4\zeta^3)$.
	\end{enumerate}
\end{lemma}

Note that the bound from (ii) is stronger than that from (i) for $l = 0$. Indeed, it relies on a special cancellation that is false for $l > 0$. This cancellation is an essential ingredient of our proof, and our ability to exploit it hinges on the identification of $Q$ as a central object of our analysis. As it turns out, we are able to derive bounds on $Q$ that are closed in the sense that they are self-improving; this allows us to iterate these bounds and hence obtain the optimal bounds on $Q$ stated in Lemma \ref{lem:estimation_of_Q}.

We postpone the proof of Lemma \ref{lem:estimation_of_Q} and turn to complete the proof of Lemma \ref{prop:main_estimations_averaging}.

\begin{proof}[Proof of Lemma \ref{prop:main_estimations_averaging}]
	Using \eqref{eq:Y_k} and Lemma \ref{lem:cumulant_factos_estimate}, for every fixed $k$ we have 
	\begin{equation*}
	|Y_k| = O(N^{-\frac{k-1}{2}})\cdot \frac{1}{N^2}\sum_{i,j} \E\big|\partial_{ij}^k ((GB)_{ji}\underline{\cal D B}^{p-1} \overline{\underline{\cal D B}}^p)\big| \,.
	\end{equation*}
	The last expression can be estimated by a sum, over $r,s,t\geq 0$ such that $r+s+t=k$, of terms of the form 
	\begin{equation}\label{eq:r_s_t}
	O(N^{-\frac{k-1}{2}})\cdot \frac{1}{N^2}\sum_{i,j} \E\absb{(\partial_{ij}^r (GB)_{ji})(\partial_{ij}^s\underline{\cal D B}^{p-1})(\partial_{ij}^t \overline{\underline{\cal D B}}^p)}\,.
	\end{equation}
	To simplify notation, here we ignore the complex conjugate on $\ul{\cal D B}$, which plays no role in the following analysis and represents a trivial notational complication, and estimate
	\begin{equation}\label{eq:r_s_t_simple}
	O(N^{-\frac{k-1}{2}})\cdot \frac{1}{N^2}\sum_{i,j} \E\absb{(\partial_{ij}^r (GB)_{ji})(\partial_{ij}^{k-r}\underline{\cal D B}^{2p-1})}
	\end{equation}
	for $r = 0, \dots, k$. 
	
	Computing the derivative $\partial_{ij}^{k-r}$, we find that \eqref{eq:r_s_t_simple} is bounded by a sum of terms of the form
	\begin{equation} \label{eq:r_s_t_2}
	O(N^{-\frac{k-1}{2}})\cdot \frac{1}{N^2}\sum_{i,j}\E\,\bigg|(\partial^r_{ij}(GB)_{ji}) \underline{\cal D B}^{2p-1-q} \prod_{m=1}^q (\partial_{ij}^{l_m}\underline{\cal D B}) \bigg|\,,
	\end{equation}
	where the sum ranges of integers $q = 0, \dots, (k-r) \wedge (2p - 1)$ and $l_1, \dots, l_q \geq 1$ satisfying $l_1 + \dots + l_q = k-r$.
	Using \eqref{eq:relation_of_der_avg_DB_and_Q} to rewrite the derivative of $\underline{\cal D B}$ and noting that $\partial_{ij}^l((GB)_{ji}) \prec (1+\phi)^{l+1}$ for $l\geq 0$ by \eqref{eq:bound_on_G}, we conclude that \eqref{eq:r_s_t_2} is stochastically dominated by 
	\begin{equation*}
	O(N^{-\frac{k-1}{2}})(1+\phi)^{r+1}\cdot \frac{1}{N^2}\sum_{i,j} \E \qBB{\bigg|\prod_{m=1}^q\big[\abs{\partial_{ij}^{l_m-1}(Q_{ij}+Q_{ji})}+N^{-1}(1+\phi)^{l_m}\big]\bigg|\cdot|\underline{\cal D B}|^{2p-1-q}}\,.
	\end{equation*}
	Using Lemma \ref{lem:estimation_of_Q} (i) to estimate the derivatives of $Q_{ij}$ and $Q_{ji}$, we conclude
	\begin{align*}
	|Y_k| &= \sum_{q=0}^{k\wedge (2p-1)} O(N^{-\frac{k-1}{2}})\cdot O_\prec((1+\phi)^{k+1}\zeta^{2q})\cdot \E|\underline{\cal D B}|^{2p-1-q}\\
	&\leq \sum_{q=0}^{2p-1}  O_\prec(N^{-\frac{k-3}{2}}(1+\phi)^{k+1}\zeta^{2q+2})\cdot \E|\underline{\cal D B}|^{2p-1-q}\,.
	\end{align*}
	Noting that $N^{-\frac{k-3}{2}}(1+\phi)^{k+1}\leq (1+\phi)^4$ whenever $k\geq 3$, we conclude the proof for the case $k\geq 3$.
	
	In the remaining cases for Lemma \ref{prop:main_estimations_averaging} (i) and (ii), that is $k\in\{1,2\}$, this rough argument is not precise enough, and one needs to obtain an additional factor of $\zeta$. As it turns out, those terms are the ones in which $r=1$. This can be done by a more careful analysis, which we now perform.
	
	\subsubsection*{(i) The case $k=1$} This case makes use of an algebraic cancellation, and we therefore track the complex conjugates carefully. Using that $\cal C_2(H_{ij}) = \E[H_{ij}^2]$, the formula for the derivative of $G$ (see \eqref{3.15}), and the one for the derivative of $\underline{\cal D B}$ (see \eqref{eq:relation_of_der_avg_DB_and_Q}), one can verify that 
	\begin{equation}\label{k=1_cancel}
	\begin{aligned}
	&(1+\delta_{ij})(\E[\underline{\cal K B}\cdot \underline{\cal D B}^{p-1}\overline{\underline{\cal D B}}^p]+Y_1) \\
	&\qquad = -(1+\delta_{ij})\E[\underline{\cal J B}\cdot  \underline{\cal D B}^{p-1} \overline{\underline{\cal D B}}^p] \\
	&\qquad +\frac{(p-1)}{N^3}\sum_{i,j}s_{ij}\E\big[((GB)_{ji}(GB)_{ij}+(GB)_{ji}(GB)_{ji})\underline{\cal D B}^{p-2}\overline{\underline{\cal D B}}^p\big]\\
	&\qquad +\frac{p}{N^3}\sum_{i,j}s_{ij}\E\big[((GB)_{ji}(\overline{GB})_{ij}+(GB)_{ji}(\overline{GB})_{ji})\underline{\cal D B}^{p-1}\overline{\underline{\cal D B}}^{p-1}\big]\\
	&\qquad - \frac{p-1}{N^2} \sum_{i,j} s_{ij} \E\big[(GB)_{ji}(Q_{ij}+Q_{ji})\underline{\cal D B}^{p-2}\overline{\underline{\cal D B}}^p\big]\\
	&\qquad - \frac{p}{N^2} \sum_{i,j} s_{ij} \E\big[(GB)_{ji}(\overline{Q}_{ij}
	+\overline{Q}_{ji})\underline{\cal D B}^{p-1}\overline{\underline{\cal D B}}^{p-1}\big]\,.
	\end{aligned}
	\end{equation}
	Note that on the right-hand side of this identity a crucial cancellation took place: the first line comes from the sum of the term $\underline{\cal K B}$ and the term $r=1$ in $Y_1$, which almost cancel, yielding the small term $\underline{\cal J B}$.
	
	Taking an absolute value, using \eqref{egWard2} for the first term, the Ward identity and the fact that $s_{ij} = O(1)$ for the second and third terms and Lemma \ref{lem:estimation_of_Q} (ii) together with the Ward identity for the last two terms, the claim follows.
	
	\subsubsection*{(ii) The case $k=2$} From Lemma \ref{lem:cumulant_factos_estimate} we get $|\cal C_3(H_{ij})|=O(N^{-3/2})$, and therefore
	\begin{equation} \label{Y2}
	|Y_2|=O(N^{-5/2})\cdot \sum_{i,j}\,\E \big|\partial_{ij}^2\big((GB)_{ji} \underline{\cal D B}^{p-1}\overline{\underline{\cal D B}}^p\big)\big|\,.
	\end{equation}
	As before, we ignore the complex conjugates to simplify notation, and estimate
	\begin{equation} \label{Y21}
	|Y_2|=O(N^{-5/2})\cdot \sum_{i,j}\,\E \big|\partial_{ij}^2\big((GB)_{ji} \underline{\cal D B}^{2p-1}\big)\big|\,.
	\end{equation}
	instead of \eqref{Y2}. The proof is performed using an explicit computation of the second derivative
	\begin{multline*}
	\partial_{ij}^2\big((GB)_{ji} \underline{\cal D B}^{2p-1}\big) = 
	(\partial_{ij}^2 (GB)_{ji}) \underline{\cal D B}^{2p-1}
	+ 2(2p-1)(\partial_{ij}(GB)_{ji})(\partial_{ij}\underline{\cal D B})\underline{\cal D B}^{2p-2} \\
	+(2p-1)(GB)_{ji}(\partial_{ij}^2 \underline{\cal D B}) \underline{\cal D B}^{2p-2}
	+(2p-1)(2p-2)(GB)_{ji}(\partial_{ij}\underline{\cal D B})^2 \underline{\cal D B}^{2p-3}\\
	\eqd  (1+\delta_{ij})^{-2}(Z^1_{ij} \, \underline{\cal D B}^{2p-1} +Z^2_{ij} \, \underline{\cal D B}^{2p-2}+Z^3_{ij} \, \underline{\cal D B}^{2p-3})\,,
	\end{multline*}
	in self-explanatory notation.
	
	Starting with $Z^1_{ij}$, we have 
	\begin{equation*}
	(1+\delta_{ij})^2 Z^1_{ij} = 2G_{ji}G_{ji}(GB)_{ji} + 2G_{jj}G_{ii}(GB)_{ji}+2G_{ji}G_{jj}(GB)_{ii}+2G_{jj}G_{ij}(GB)_{ii}\,.
	\end{equation*}
	Using the bound $N^{-1/2} \leq \zeta$, the Ward identity, and \eqref{eq:bound_on_G} on each of the four terms, we obtain
	\begin{equation*}
	O(N^{-5/2})\sum_{i,j}\E\big|Z^1_{ij}\underline{\cal D B}^{2p-1}\big| \prec O_\prec((1+\phi)^2\zeta^2)\E|\underline{\cal D B}|^{2p-1}\,,
	\end{equation*}
	as desired.
	
	The first term of $Z^2_{ij}$ yields the contribution
	\begin{multline*}
	O(N^{-5/2})\sum_{i,j}\E\big|(\partial_{ij}(GB)_{ji})(\partial_{ij}\underline{\cal D B})\underline{\cal D B}^{2p-2}\big|\\
	=O(N^{-5/2}) \sum_{i,j}\E\big|(\partial_{ij}(GB)_{ji})\big(N^{-1}(GB)_{ij}+N^{-1}(GB)_{ji}-Q_{ij}-Q_{ji}\big)\underline{\cal D B}^{2p-2}\big|\,,
	\end{multline*}
	where we used \eqref{eq:relation_of_der_avg_DB_and_Q}. Using \eqref{eq:bound_on_G}, \eqref{3.15}, the bound $N^{-1/2} \leq \zeta$, and the Ward identity, one can estimate the contribution of the first two terms by $O_\prec((1+\phi)^2\zeta^4)\E|\underline{\cal D B}|^{2p-2}$. Using \eqref{eq:bound_on_G}, Lemma \ref{lem:estimation_of_Q} (ii), and the Ward identity, one can bound the contribution of the last two terms by $O_\prec((1+\phi)^6 \zeta^4)\E|\underline{\cal D B}|^{2p-2}$, as desired.
	
	In the same spirit, using Lemma \ref{lem:estimation_of_Q} (i), the second term of $Z_{ij}^2$ is of the form 
	\begin{align*}
	&O(N^{-5/2})\sum_{i,j}\E\big|(GB)_{ji}(\partial_{ij}^2 \underline{\cal D B}) \underline{\cal D B}^{2p-2}\big|\\
	&\quad =O(N^{-5/2}) \sum_{i,j}\E\big|(GB)_{ji}\big(\partial_{ij}\big(N^{-1}((GB)_{ij}+(GB)_{ji})-(Q_{ij}+Q_{ji})\big)\big)\underline{\cal D B}^{2p-2}\big|\\
	&\quad =  O(N^{-1/2})\cdot \frac{1}{N^2}\sum_{i,j}\E\big|(GB)_{ji} O_\prec((1+\phi)^2\zeta^2) \underline{\cal D B}^{2p-2}\big|=O((1+\phi)^2\zeta^4)\E|\underline{\cal D B}|^{2p-2}\,.
	\end{align*}
	
	Finally, the contribution of $Z_{ij}^3$ is of the form
	\begin{equation*}
	O(N^{-5/2})\sum_{i,j}\E\big|(GB)_{ji}(\partial_{ij}\underline{\cal D B})^2 \underline{\cal D B}^{2p-3}\big|\,,
	\end{equation*}  
	using Lemma \ref{lem:estimation_of_Q} (i) to bound $(\partial_{ij}\underline{\cal D B})^2 = O_\prec((1+\phi)^4\zeta^4)$ and using the Ward identity for $(GB)_{ji}$, we conclude that the last term equals $O_\prec((1+\phi)^4\zeta^6)$, as desired. This concludes the proof of Lemma \ref{prop:main_estimations_averaging} (i) and (ii).
	
	\subsubsection*{(iii) The remainder term}  The proof is similar to that of Lemma \ref{prop:main_estimations} (iii), and we omit the details.
\end{proof}

\subsection{Proof of Lemma \ref{lem:estimation_of_Q}}

We start with several observations regarding the matrix $Q$. Define the random matrix
\begin{equation*}
F_{ab}\equiv F(H)_{ab}\deq \frac{1}{N}(BH)_{ab} + \frac{\delta_{ab}}{N^2}\sum_{c} (GB)_{\f e_c^b c} + \frac{1}{N^2}B_{ab}\sum_c G_{\f e_c^b c}\, .
\end{equation*}
Thus we have
\begin{equation*}
Q = GFG\, ,	
\end{equation*}
and therefore by \eqref{3.15} we have for all $\f v,\f w\in\mathbb{S}$
\begin{equation}\label{eq:der_Q}
(1+\delta_{ij})\partial_{ij} Q_{\f v\f w} = -(G_{\f v i}Q_{j\f w} + G_{\f v j}Q_{i\f w} + Q_{\f v i}G_{j\f w} + Q_{\f v j}G_{i\f w}) + E^{ij}_{\f v\f w}\,,
\end{equation}
where $E^{ij} \deq  (1+\delta_{ij})G(\partial_{ij}F)G$.

\begin{proof}[Proof of Lemma \ref{lem:estimation_of_Q} (i)]
	We prove the claim by induction on $l$. We start with the case $l = 0$. Fix $\f v,\f w\in\mathbb{S}$. Using  \eqref{eq:defn_of_Q}, \eqref{eq:bound_on_G}, and the Ward identity, we obtain 
	\begin{equation*}
	\begin{aligned}
	|Q_{\f v\f w}| &\leq  \frac{1}{N} |(GBHG)_{\f v\f w}| + \frac{1}{N^2}\sum_{a,b} \big|G_{\f va}G_{a\f w}(GB)_{\f e_b^a b}\big|+\frac{1}{N^2}\sum_{a,b} \big|G_{aa}(GB)_{\f v b}G_{\f e_b^a  \f w}\big|\\
	& \prec \frac{1}{N}\sqrt{|(GBB^*G^*)_{\f v\f v}|}\sqrt{|(G^*H^*HG)_{\f w\f w}|} + (1+\phi)\zeta^2 \,.
	\end{aligned}
	\end{equation*}
Using the bound \eqref{eq:bound_on_G} the Ward identity, and Lemma \ref{lem:H_bound}, we therefore get
	\begin{equation*}
	\frac{1}{N}|(G^*H^*HG)_{\f w\f w}| \leq \frac{1}{N}\|H\|^2(G^*G)_{\f w\f w} = \|H\|^2 \frac{\im G_{\f w\f w}}{N\eta} \leq \norm{H}^2\zeta^2 \prec \zeta^2\,.
	\end{equation*}
	Similarly, using the assumption $\norm{B}=O(1)$ we find $\frac{1}{N} |(GBB^*G^*)_{\f v\f v}| \leq O(\zeta^2)$.
	We conclude that $|Q_{\f v\f w}| \prec (1+\phi)\zeta^2$, which completes the proof of the case $l = 0$.
	
	We now perform the induction step.
	Suppose that $l \geq 1$ and $\partial_{ij}^{m}Q = O_\prec((1+\phi)^{m+1}\zeta^2)$ for $0\leq m\leq l-1$. By \eqref{eq:der_Q} we have
	\begin{equation} \label{der_Q}
	(1+\delta_{ij})\partial_{ij}^l Q_{\f v\f w} = -\partial_{ij}^{l-1}(G_{\f v i}Q_{j\f w} + G_{\f v j}Q_{i\f w} + Q_{\f v i}G_{j\f w} + Q_{\f v j}G_{i\f w}) + \partial_{ij}^{l-1}E_{\f v\f w}^{ij}\,.
	\end{equation}
	We deal with each of the terms on the right-hand side separately. The term $\partial_{ij}^{l-1}(G_{\f vi}Q_{j\f w})$ is estimated using \eqref{eq:bounds_on_derivatives_of_G} as
	\begin{multline*}
	\absb{\partial_{ij}^{l-1}(G_{\f vi}Q_{j\f w})} \leq \sum_{m=0}^{l-1} \binom{l-1}{m} \absb{(\partial_{ij}^m G_{\f v i}) (\partial^{l-1-m}_{ij}Q_{j\f w})} \prec \sum_{m=0}^{l-1}(1+\phi)^{m+1} \absb{\partial^{l-1-m}_{ij}Q_{j\f w}}\\
	\prec \sum_{m=0}^{l-1}(1+\phi)^{m+1} (1+\phi)^{l-m}\zeta^2 = O((1+\phi)^{l+1}\zeta^2)\,,
	\end{multline*}
	where in the third step we used the induction assumption. The three subsequent terms of \eqref{der_Q} are estimated analogously.
	
	In order to estimate the last term of \eqref{der_Q}, we write
	\begin{equation}\label{eq:formula_for_E}
	\begin{aligned}
	E^{ij}_{\f v\f w} &= (1+\delta_{ij})\sum_{a,b}G_{\f v a}(\partial_{ij}F_{ab})G_{b\f w}\\
	& = \frac{1}{N} (GB)_{\f v i}G_{j\f w} +\frac{1}{N} (GB)_{\f v j}G_{i\f w} \\
	& \quad - \frac{1}{N^2} \sum_{a,b}G_{\f v a}G_{a\f w} G_{\f e_b^ai}(GB)_{jb} -
	\frac{1}{N^2} \sum_{a,b}G_{\f v a}G_{a\f w} G_{\f e_b^a j}(GB)_{ib}\\
	& \quad -\frac{1}{N^2} \sum_{a,b} (GB)_{\f v a}G_{a\f w} G_{\f e_b^a i}G_{jb}
	-\frac{1}{N^2} \sum_{a,b} (GB)_{\f v a} G_{a\f w} G_{\f e_b^a j}G_{i b}\,.
	\end{aligned}
	\end{equation}
	Using \eqref{3.15}, the assumption $\norm{B}=O(1)$ and the bound \eqref{eq:bound_on_G} it follows that
	\begin{equation*}
	\partial_{ij}^{l-1}\bigg(\frac{1}{N} (GB)_{\f v i}G_{j\f w} +\frac{1}{N} (GB)_{\f v j}G_{i\f w}\bigg) \prec \frac{1}{N}(1+\phi)^{l+1}\leq (1+\phi)^{l+1}\zeta^2. 
	\end{equation*}
	What remains, therefore, is to estimate $\partial_{ij}^{l-1}$ applied to the four last terms of \eqref{eq:formula_for_E}. They are all similar, and we consider $\partial_{ij}^{l-1}\Big(\frac{1}{N^2} \sum_{a,b}G_{\f v a}G_{a\f w} G_{\f e_b^ai}(GB)_{jb}\Big)$ for definiteness. By computing the derivatives and using \eqref{3.15}, we obtain a sum of terms, each of which is a product of $l+3$ factors of the form $G_{\f x\f y}$ for some $\f x,\f y\in \{\f v,\f w,\f e_a,\f e_b,\f e_b^a,\f e_i,\f e_j\}$ and one factor of the form $(GB)_{\f x b}$ for some $\f x\in \{i,j\}$. Furthermore, exactly two of those factors are of the form $G_{\f x a}$ or $G_{a\f x}$ with $\f x\in \{\f v,\f w,\f e_i,\f e_j\}$. Applying the Ward identity with the sum over $a$ to the two factors containing an index $a$, and estimating the remaining $l + 1$ factors using \eqref{eq:bound_on_G}, we conclude that 
	\begin{equation*}
	\partial_{ij}^{l-1}\bigg(\frac{1}{N^2} \sum_{a,b}G_{\f v a}G_{a\f w} G_{\f e_b^ai}(GB)_{jb}\bigg) \prec (1+\phi)^{l+1}\zeta^2\,,
	\end{equation*}
	as desired. This concludes the proof.
\end{proof}

\begin{proof} [Proof of Lemma \ref{lem:estimation_of_Q} (ii)]
	We have to improve the naive bound $Q = O_\prec((1 + \phi) \zeta^2)$ from part (i) by an order $\zeta$. We do this by deriving a stochastic self-improving estimate on $Q$, which makes use of a crucial cancellation analogous to, but more subtle than, the one in \eqref{k=1_cancel}. We derive this self-improving bound using another high-moment estimate for the entries of $Q$, which is derived using the cumulant expansion from Lemma \ref{lem:cumulant_expansion}.

	To this end, fix $\f v,\f w\in \mathbb{S}$ and $p\in\N$. From \eqref{eq:defn_of_Q} we immediately get
	\begin{equation*}
	Q_{\f v\f w} = \frac{1}{N} (GBHG)_{\f v\f w} + \frac{1}{N^2}\sum_{i,j} G_{\f v j}G_{j\f w}(GB)_{\f e_i^j i}+\frac{1}{N^2}\sum_{i,j} G_{jj}(GB)_{\f v i}G_{\f e_i^j \f w}\,.
	\end{equation*}
	Thus,
	\begin{equation*}
	\E|Q_{\f v\f w}|^{2p} = \E\Big[\cal W_{\f v\f w}Q_{\f v \f w}^{p-1}\overline{Q}_{\f v \f w}^{p}\Big]+ \frac{1}{N}\sum_{i,j}\E\big[(GB)_{\f v i}H_{ij}G_{j\f w}Q_{\f v \f w}^{p-1}\overline{Q}_{\f v \f w}^{p}\big]\,,
	\end{equation*}
	where we defined
	\begin{equation*}
	\cal W_{\f v\f w} \deq\frac{1}{N^2}\sum_{i,j} G_{\f v j}G_{j\f w}(GB)_{\f e_i^j i}+\frac{1}{N^2}\sum_{i,j} G_{jj}(GB)_{\f v i}G_{\f e_i^j \f w}\,.
	\end{equation*}
	Using the cumulant expansion (Lemma \ref{lem:cumulant_expansion}) for the last term on the right-hand side with $h=H_{ij}$ and $f(H_{ij}) = (GB)_{\f v i}G_{j\f w}Q_{\f v \f w}^{p-1}\overline{Q}_{\f v \f w}^{p}$, we obtain 
	\begin{equation}\label{eq:power_of_Q}
	\E|Q_{\f v\f w}|^{2p} = \E\Big[\cal W_{\f v\f w}Q_{\f v \f w}^{p-1}\overline{Q}_{\f v \f w}^{p}\Big] + \sum_{k=1}^{\ell} Z_k + \sum_{i,j} \hat{\cal R}_{\ell+1}^{ij}\, ,
	\end{equation}
	where we used the notation 
	\begin{equation*}
	Z_k = \frac{1}{N}\sum_{i,j}\frac{1}{k!}\cal C_{k+1}(H_{ij})\E\big[\partial_{ij}^k\big((GB)_{\f v i}G_{j\f w}Q_{\f v \f w}^{p-1}\overline{Q}_{\f v \f w}^{p}\big)\big] 
	\end{equation*}
	Here $\ell$ is a fixed positive integer to be chosen later, and $\hat {\cal R}_{\ell+1}^{ij}$ is a remainder term defined analogously to $\cal R^{ij}_{\ell+1}$ in (\ref{tau1}).

	The proof of Lemma \ref{lem:estimation_of_Q} can be broken into the following estimates.
	
	\begin{lemma}\label{lem:cumulant_expansion_for_Q_estimates}
		Suppose that $Q = O_\prec(\lambda)$ for some $\lambda\geq (1+\phi)^4\zeta^3$.
		\begin{enumerate}
			\item We have
			\begin{equation*}
			\E\big[\cal W_{\f v\f w}Q_{\f v \f w}^{p-1}\overline{Q}_{\f v \f w}^{p}\big]+Z_1 = O_\prec(\zeta^3 )\E|Q_{\f v\f w}|^{2p-1}+O_\prec(\zeta^3 \lambda)\E|Q_{\f v\f w}|^{2p-2}
			\end{equation*}
			\item For $k\geq 2$ we have
\begin{equation} \label{Z_k_est_2}
Z_k = \sum_{n=1}^{2p}(1+\phi)^{2n} O_\prec\pb{(\zeta^3 \lambda)^{n/2}} \E|Q_{\f v\f w}|^{2p-n}\,.
\end{equation}
			\item For any $D>0$, there exists $\ell\equiv\ell(D)\geq 1$  such that $\hat {\cal R}_{\ell+1}^{ij}=O(N^{-D})$ uniformly for all $i,j\in \qq{N}$.
		\end{enumerate}
	\end{lemma}
	
	We now conclude the proof of Lemma \ref{lem:estimation_of_Q} (ii) using Lemma \ref{lem:cumulant_expansion_for_Q_estimates}. 
	Suppose that $Q=O_\prec(\lambda)$ for some $\lambda\geq \zeta^3$. Combining Lemma \ref{lem:cumulant_expansion_for_Q_estimates} for $\ell=\ell(6p+2)$ together with \eqref{eq:power_of_Q} we obtain 
	\begin{equation*}
	\E |Q_{\f v\f w}|^{2p}=  \sum_{n=1}^{2p}  O_{\prec}\big((1+\phi)^{2n}(\zeta^{3}\lambda)^{n/2}\big)\cdot\E|Q_{\f v\f w}|^{2p-n}
	+O(N^{-6p})\,.
	\end{equation*}
	Since $\zeta \geq N^{-1/2}$, it follows from the H\"older's inequality that 
	\begin{equation*}
	\E |Q_{\f v\f w}|^{2p}\leq 
    \sum_{n=1}^{2p}  O_{\prec}\big((1+\phi)^{2n}(\zeta^{3}\lambda)^{n/2}\big)\cdot(\E|Q_{\f v\f w}|^{2p})^{\frac{2p-n}{2p}}\,,
	\end{equation*}
	and therefore by Young's inequality 
	\begin{align*}
	\E |Q_{\f v\f w}|^{2p}= O_\prec\big(((1+\phi)^4\zeta^3 \lambda)^p\big)\,.
	\end{align*}
	
	Next, since $p$ was arbitrary, we conclude from Markov's inequality that
	\begin{equation} \label{Q_self_improv}
	Q = O_\prec(\lambda) \qquad \Longrightarrow \qquad Q=O_\prec\pB{\pb{(1+\phi)^4 \zeta^3 \lambda}^{1/2}}
	\end{equation}
	for any $\lambda \geq (1+\phi)^4\zeta^3$.
	Moreover, by Lemma \ref{lem:estimation_of_Q} (i) we have the a priori bound $Q=O_\prec((1+\phi)\zeta^2)$. An iteration of \eqref{Q_self_improv}, analogous to Lemma \ref{lem:self-improv}, yields $Q = O_\prec((1 + \phi)^4 \zeta^3)$, as claimed.
\end{proof}

What remains, therefore, is the proof of Lemma \ref{lem:cumulant_expansion_for_Q_estimates}.

\begin{proof}[Proof of Lemma \ref{lem:cumulant_expansion_for_Q_estimates}] $~$
	We begin with (i). A straightforward computation yields
	\begin{equation}\label{eq:whatever}
	\begin{aligned}
	& \E\Big[\cal W_{\f v\f w}Q_{\f v \f w}^{p-1}\overline{Q}_{\f v \f w}^{p}\Big]+Z_1 \\
	&\quad = -\frac{1}{N^2}\sum_{i,j}\E\Big[G_{\f v \f e_i^j}(GB)_{ji}G_{j\f w}Q_{\f v\f w}^{p-1}\overline{Q}_{\f v \f w}^p\Big] - \frac{1}{N^2}\sum_{i,j}\E\Big[(GB)_{\f v \f e_i^j}G_{ji}G_{j\f w}Q_{\f v\f w}^{p-1}\overline{Q}_{\f v\f w}^p\Big]\\
	&\quad +\frac{p-1}{N^2} \sum_{i,j} \E \Big[(GB)_{\f v \f e_i^j}G_{j\f w}(\partial_{ij}Q_{\f v\f w})Q_{\f v\f w}^{p-2}\overline{Q}_{\f v\f w}^{p}\Big] +\frac{p}{N^2} \sum_{i,j}\E \Big[(GB)_{\f v \f e_i^j}G_{j\f w}(\partial_{ij}\overline{Q}_{\f v\f w})Q_{\f v\f w}^{p-1}\overline{Q}_{\f v\f w}^{p-1}\Big]\,.
	\end{aligned}
	\end{equation}
	We emphasize that here an important cancellation takes place between the two terms on the left-hand side, whereby two large terms arising from the computation of $Z_1$ precisely cancel out the two terms of $\cal W_{\f v \f w}$.
	Using the Ward identity and $s_{ij} = O(1)$, one can verify that the absolute value of the first two terms on the right-hand side of \eqref{eq:whatever} is $O_\prec(\zeta^3) \E \abs{Q_{\f v \f w}}^{2p - 1}$. Hence, it suffices to show that the remaining two terms are $O_\prec(\zeta^3\lambda)  \E \abs{Q_{\f v \f w}}^{2p - 2}$. Taking for example the third term on the right-hand side of \eqref{eq:whatever}, by \eqref{eq:der_Q} we find
	\begin{multline*}
	\bigg|\frac{1}{N^2}\sum_{i,j}\E \Big[(GB)_{\f v \f e_i^j}G_{j\f w}(\partial_{ij}Q_{\f v\f w})Q_{\f v\f w}^{p-2}\overline{Q}_{\f v \f w}^p\Big]\bigg|\\
	\leq \frac{1}{N^2}\sum_{i,j}\E\Big[\big| (GB)_{\f v\f e_i^j}G_{j\f w}(G_{\f v i}Q_{j\f w} + G_{\f v j}Q_{i\f w} + Q_{\f v i}G_{j\f w}+Q_{\f v j}G_{i\f w})\big|\cdot |Q_{\f v \f w}|^{2p-2}\Big]\\
	+ \frac{1}{N^2}\sum_{i,j}\E\big[\big|(GB)_{\f v \f e_i^j}G_{j\f w}E_{\f v\f w}^{ij}\big|\cdot |Q_{\f v \f w}|^{2p-2}\big]\,.
	\end{multline*}
	Due to the assumption $G=O_\prec(\lambda)$ and using the Ward identity, one can bound the first term on the right-hand side by $O_\prec(\zeta^3 \lambda)\E|Q_{\f v \f w}|^{2p-2}$. As for the last term, using \eqref{eq:formula_for_E} and the Ward identity on the indices $i,j$ and when possible also $a,b$, we find that it is $O_\prec(\zeta^6)\E|Q_{\f v \f w}|^{2p-2}$. This conclude the proof of part (i).
	
	Next, we prove (ii). Fix $k\geq 2$ and estimate
	\begin{equation*}
	\begin{aligned}
	|Z_k| &= O(N^{-\frac{k-1}{2}}) \cdot \frac{1}{N^2}\sum_{i,j}\E\big|\partial_{ij}^k\big((GB)_{\f v i}G_{j\f w}Q_{\f v \f w}^{p-1}\overline{Q}_{\f v \f w}^{p}\big)\big| \\
	& = O(N^{-\frac{k-1}{2}}) \sum_{\substack{r,s,t\geq 0 \\ r+s+t=k}} \frac{1}{N^2}\sum_{i,j}\E\big|(\partial_{ij}^r((GB)_{\f v i}G_{j
		\f w}))(\partial_{ij}^s Q_{\f v\f w}^{p-1})(\partial_{ij}^t \overline{Q}_{\f v\f w}^p)\big|\,.
	\end{aligned}
	\end{equation*}
	As the sum over $r,s,t$ is finite it suffices to deal with each term separately. To simplify notation, we drop the complex conjugates of $Q$ (which play no role in the subsequent analysis), and estimate the quantity
	\begin{equation}\label{eq:whatever_2}
	O(N^{-\frac{k-1}{2}}) \cdot \frac{1}{N^2}\sum_{i,j}\E\big|\big(\partial_{ij}^r((GB)_{\f v i}G_{j
		\f w})\big)\big(\partial_{ij}^{k-r} Q_{\f v\f w}^{2p-1}\big)\big|
	\end{equation}
	for $r = 0, \dots, k$.
    Computing the derivative $\partial_{ij}^{k - r}$, we find that \eqref{eq:whatever_2} is bounded by a sum of terms of the form
	\begin{equation} \label{sql}
	O(N^{-\frac{k-1}{2}})  \frac{1}{N^2}\sum_{i,j}\E\,\Big|\big(\partial_{ij}^r((GB)_{\f v i}G_{j
		\f w})\big) \Big(\prod_{m=1}^{q}\big(\partial_{ij}^{l_m}Q_{\f v\f w})\Big)Q_{\f v\f w}^{2p-1-q}\Big|\,,
	\end{equation}
	where the sum ranges over integers $q = 0, \dots, (k - r) \wedge (2p - 1)$ and $l_1, \dots, l_q \geq 1$ satisfying $l_1 + \cdots + l_q = k - r$.
	Using Lemma \ref{lem:estimation_of_Q} (i), we find that \eqref{sql} is bounded by
	\begin{equation*}
	O_\prec(N^{-\frac{k-1}{2}}) \frac{1}{N^2}\sum_{i,j}\E\big[\big|\partial_{ij}^r((GB)_{\f v i}G_{j
		\f w})\big| (1+\phi)^{k-r+q}\zeta^{2q} |Q_{\f v\f w}|^{2p-1-q}\big]\,.
	\end{equation*}
	
	Note that by \eqref{3.15} the derivative $\partial_{ij}^r((GB)_{\f v i}G_{j\f w})$ can be written as a sum of terms, each of which is a product of $r+2$ entries of the matrices $GB$ or $G$, with one entry of the form $(GB)_{\f v a}$ or $G_{\f v a}$ with $a\in \{i,j\}$ and one of the form $G_{a\f w}$ with $a\in \{i,j\}$. Using the Ward identity for the two specified terms in the product and using \eqref{eq:bound_on_G} to bound the remaining terms in the product, we conclude that $\frac{1}{N^2}\sum_{i,j}\big|(\partial_{ij}^r((GB)_{\f v i}G_{j\f w})\big|\prec (1+\phi)^{r}\zeta^2$, and therefore 
	\begin{align}
	\eqref{sql} &\prec O(N^{-\frac{k-1}{2}})  (1+\phi)^{k+q}\zeta^{2q+2}\E|Q_{\f v\f w}|^{2p-1-q}
	\notag \\
	& = O(N^{-\frac{k-1}{2}}\zeta^{-q-1}(1+\phi)^{k+q})\zeta^{3q+3}\E|Q_{\f v\f w}|^{2p-1-q}
	\notag \\ \label{sro}
	&\leq O(\zeta^{k-2-q}(1+\phi)^{q+1})\zeta^{3q+3}\E|Q_{\f v\f w}|^{2p-1-q}\,,
	\end{align}
	where for the last inequality we used $N^{-1/2}(1+\phi)\leq \zeta$. Clearly, if $q\leq k-2$ then \eqref{sro} is bounded by $O_{\prec}((1+\phi)^{q+1}\zeta^{3q+3})\cdot \E |Q_{\f v \f w}|^{2p-1-q}$, as desired.
	
	What remains, therefore, is to estimate \eqref{sql}  for $q \geq k - 1$, which we assume from now on.
	Because $k \geq 2$ by assumption, we find that $q \geq 1$. Moreover, since $q = 0, \dots, (k - r) \wedge (2p - 1)$, we find that $r \leq 1$. Thus, it remains to consider the three cases  $(r,q)=(0,k)$, $(r,q)=(1,k-1)$ and $(r,q)=(0,k-1)$. We deal with them separately.
	
	\paragraph{The case $(r,q)=(0,k)$} In this case $l_1=l_2=\dots=l_q=1$, so that \eqref{sql} reads
	\begin{multline}\label{eq:whatever_3}
	O(N^{-\frac{k-1}{2}}) \cdot \frac{1}{N^2} \sum_{i,j} \E\big|(GB)_{\f v i}G_{j\f w} (\partial_{ij}Q_{\f v\f w})^k Q_{\f v\f w}^{2p-1-k}\big|\\
	\leq  O(N^{-\frac{k-1}{2}}) \cdot \frac{1}{N^2} \sum_{i,j} \E\big|(GB)_{\f v i}G_{j\f w} (G_{\f v i}Q_{j\f w} + G_{\f v j}Q_{i\f w} + Q_{\f v i}G_{j\f w} + Q_{\f v j}G_{i\f w} )(\partial_{ij}Q_{\f v\f w})^{k-1} Q_{\f v\f w}^{2p-1-k}\big|\\
	+  O(N^{-\frac{k-1}{2}}) \cdot \frac{1}{N^2} \sum_{i,j} \E\big|(GB)_{\f v i}G_{j\f w} E_{\f v\f w}^{ij}(\partial_{ij}Q_{\f v\f w})^{k-1} Q_{\f v\f w}^{2p-1-k}\big|\,,
	\end{multline}
	where for the inequality we used \eqref{eq:der_Q}.
	
	We estimate the second line of \eqref{eq:whatever_3} using the Ward identity on the summations over $i$ and $j$, using the bound $Q=O_\prec(\lambda)$, and the bound $\partial_{ij}Q_{\f v\f w}\prec (1+\phi)^2\zeta^2$ from Lemma \ref{lem:estimation_of_Q} (i). The result is
	\begin{equation*}
	O(N^{-\frac{k-1}{2}}) \zeta^3 \lambda (1+\phi)^{2k}\zeta^{2(k-1)}\E|Q_{\f v\f w}|^{2p-1-k} \leq (1+\phi)^{2k}\zeta^{3k}\lambda  \E|Q_{\f v\f w}|^{2p-1-k}\,,
	\end{equation*}
	as desired, where we used $N^{-1/2} \leq \zeta$.
	For the third line of \eqref{eq:whatever_3}, we use \eqref{eq:formula_for_E} and the Ward identity to obtain
	\begin{equation*}
	\frac{1}{N^2} \sum_{i,j}\big|(GB)_{\f v i}G_{j\f w}E_{\f v\f w}^{ij}\big| \prec \zeta^6\,,
	\end{equation*}
	which, together with the bound $|\partial_{ij}Q_{\f v\f w}|\prec (1+\phi)^2\zeta^2$ obtained in Lemma \ref{lem:estimation_of_Q} (i), gives the bound
	\begin{equation*}
	(1+\phi)^{2k}\zeta^{3k+3}\E|Q_{\f v\f w}|^{2p-1-k}
	\end{equation*}
	for the third line of \eqref{eq:whatever_3}. Since $\zeta^3 \leq \lambda$, this bound is good enough.
	
	\paragraph{The case $(r,q)=(1,k-1)$} In this case $l_1=l_2=\dots=l_{k-1}=1$, and \eqref{sql} reads
	\begin{equation}\label{eq:whatever_10}
	O(N^{-\frac{k-1}{2}})  \frac{1}{N^2}\sum_{i,j}\E\big|(\partial_{ij}((GB)_{\f v i}G_{j \f w}))(\partial_{ij}Q_{\f v\f w})^{k-1} Q_{\f v\f w}^{2p-k}\big|\,.
	\end{equation}
	Using \eqref{eq:der_Q} to rewrite one factor $\partial_{ij}Q_{\f v \f w}$, we conclude that \eqref{eq:whatever_10} equals
	\begin{multline*}
	O(N^{-\frac{k-1}{2}})  \frac{1}{N^2}\sum_{i,j}\E\big|(\partial_{ij}((GB)_{\f v i}G_{j \f w}))(Q_{\f v i}G_{j\f w}+Q_{\f v j}G_{i\f w} + G_{\f v i}Q_{j\f w} + G_{\f v j}Q_{i\f w})(\partial_{ij}Q_{\f v\f w})^{k-2} Q_{\f v\f w}^{2p-k}\big|\\
	+ O(N^{-\frac{k-1}{2}})  \frac{1}{N^2}\sum_{i,j}\E\big|(\partial_{ij}((GB)_{\f v i}G_{j \f w})E_{\f v\f w}^{ij}(\partial_{ij}Q_{\f v\f w})^{k-2} Q_{\f v\f w}^{2p-k}\big|\,.
	\end{multline*}
	We now apply a similar argument to the one used in the previous case $(r,q)=(0,k)$.  We use \eqref{eq:formula_for_E}, the assumption $Q=O_\prec(\lambda)$ to bound $Q$, and the Ward identity to bound the product of entries of $G$ and $(GB)$. This gives
	\begin{equation*}
	\frac{1}{N^2}\sum_{i,j}\big|(\partial_{ij}((GB)_{\f v i}G_{j \f w}))(\partial_{ij}Q_{\f v\f w})\big| \prec (1+\phi)(\lambda \zeta^2 + \zeta^5)\,.
	\end{equation*}
	Using $\zeta^3 \leq \lambda$  and the bound $|\partial_{ij}Q_{\f v\f w}|^{k-2} \prec (1+\phi)^{2k-4}\zeta^{2k-4}$ from Lemma \ref{lem:estimation_of_Q} (i), we therefore get
	\begin{equation*}
	\begin{aligned}
	\eqref{eq:whatever_10}&= O(N^{-\frac{k-1}{2}}) (1+\phi) \lambda \zeta^2 (1+\phi)^{2k-4}\zeta^{2k-4} \E|Q_{\f v\f w}|^{2p-k}\\
	&\prec (1+\phi)^{2k} \zeta^{3k-3}\lambda\E|Q_{\f v\f w}|^{2p-k}\,,
	\end{aligned}
	\end{equation*}
	as desired
	
	\paragraph{The case $(r,q)=(0,k-1)$} Since $l_1+\dots+l_{k-1}=k$ and $l_m\geq 1$ for every $m\in \{1,\dots,k-1\}$, there exists exactly one $m$ such that $l_m=2$ and the remaining $l_m$'s are $1$. Hence, \eqref{sql} reads
	\begin{multline} \label{porsas}
	O(N^{-\frac{k-1}{2}}) \cdot \frac{1}{N^2}\sum_{i,j}\E\big|(GB)_{\f v i}G_{j \f w}(\partial_{ij}^2 Q_{\f v\f w})(\partial_{ij}Q_{\f v\f w})^{k-2} Q_{\f v\f w}^{2p-k}\big|\\
	\leq 	O(N^{-\frac{k-1}{2}}) \cdot \frac{1}{N^2}\sum_{i,j}\E\big[\big|(GB)_{\f v i}G_{j \f w}(\partial_{ij}(G_{\f v i}Q_{j\f w} + G_{\f v j}Q_{i\f w} + Q_{\f v i}G_{j\f w} + Q_{\f v j}G_{i\f w}))(\partial_{ij}Q_{\f v\f w})^{k-2} Q_{\f v\f w}^{2p-k}\big|\big]\\
	+ 	O(N^{-\frac{k-1}{2}}) \cdot \frac{1}{N^2}\sum_{i,j}\E\big|(GB)_{\f v i}G_{j \f w}(\partial_{ij}E^{ij}_{\f v\f w})(\partial_{ij}Q_{\f v\f w})^{k-2} Q_{\f v\f w}^{2p-k}\big|\,,
	\end{multline}
	where we used \eqref{eq:der_Q}. We deal with each of the terms on the right-hand side separately. For the third line of \eqref{porsas}, using the derivative formulas \eqref{3.15} and \eqref{eq:formula_for_E} together with the Ward identity, we find
	\begin{equation} \label{vasikka}
	\frac{1}{N^2}\sum_{i,j} \big|(GB)_{\f v i}G_{j \f w}\partial_{ij}E_{\f v\f w}^{ij}\big| \prec (1+\phi)^3 \zeta^5 \leq (1 + \phi)^3 \lambda \zeta^2\,.
	\end{equation}
	Using \eqref{vasikka} which we can estimate the third line of \eqref{porsas} by the right-hand side of \eqref{Z_k_est_2}, as in the previous case.
	
	All four terms in the second line of \eqref{porsas} are similar, and we estimate
	\begin{multline} \label{sika}
	O(N^{-\frac{k-1}{2}}) \cdot \frac{1}{N^2}\sum_{i,j}\E\big|(GB)_{\f v i}G_{j \f w}(\partial_{ij}G_{\f v i}Q_{j\f w})(\partial_{ij}Q_{\f v\f w})^{k-2} Q_{\f v\f w}^{2p-k}\big|\\
	= O(N^{-\frac{k-1}{2}}) \cdot \frac{1}{N^2}\sum_{i,j}\E\big|(GB)_{\f v i}G_{j \f w}(-G_{\f v i}G_{ji}Q_{j\f w}-G_{\f v j}G_{ii}Q_{j\f w} + G_{\f v i}\partial_{ij}Q_{j\f w})(\partial_{ij}Q_{\f v\f w})^{k-2} Q_{\f v\f w}^{2p-k}\big|\,.
	\end{multline}
	Using the bound $Q=O_\prec(\lambda)$ to bound $Q_{j\f w}$, the bound $|\partial_{ij}Q_{\f v\f w}|\prec (1+\phi)^2\zeta^2$ from Lemma \ref{lem:estimation_of_Q} (i) and the Ward identity (which can be applied to obtain at least $\zeta^3$ in the last term and $\zeta^2$ in the first two terms, we conclude that \eqref{sika} is bounded by
	\begin{equation*}
	O(N^{-\frac{k-1}{2}})(1+\phi)^2(\zeta^2 \lambda + \zeta^5)(1+\phi)^{2k-4}\zeta^{2k-4} \E|Q_{\f v\f w}|^{2p-k}\\
	=(1+\phi)^{2k-2}\zeta^{3k-3}\lambda \E|Q_{\f v\f w}|^{2p-k}\,,
	\end{equation*}
	which is bounded by the right-hand side of \eqref{Z_k_est_2}. This concludes the proof of part (ii).

	Finally, the proof of (iii) is similar to that of Lemma \ref{prop:main_estimations} (iii), and we omit the details.
\end{proof}

\section{Proof of Theorem \ref{thm:Wigner}}

\subsection{Preliminaries}
We start with a few simple deterministic estimates.

\begin{lemma} \label{lem:m_eta}
Under the assumptions of Theorem \ref{thm:Wigner} we have $\eta = O(\im m)$.
\end{lemma}
\begin{proof}
From $\norm{A} = O(1)$, \eqref{eq_spectr_A}, and \eqref{m_varrho}, it is not hard to deduce that the support of $\varrho$ lies in a ball of radius $O(1)$ around the origin. The claim then follows from the fact that $\im m(z) = \int \frac{\eta}{(E - x)^2 + \eta^2} \nu(\dd x) \geq c \eta$, since $(E - x) = O(1)$ on the support of $\nu$.
\end{proof}

\begin{lemma} \label{lem:imM}
Under the assumptions of Theorem \ref{thm:Wigner} we have $\norm{\im M} = O(\im m)$.
\end{lemma}
\begin{proof}
By the resolvent identity we have
\begin{equation*}
\norm{\im M} = \norm{M - M^*}/2 = \norm{MM^*} \im (z + m) =\norm{MM^*} (\eta+\im  m)\,,
\end{equation*}
and the claim follows by Lemma \ref{lem:m_eta}.
\end{proof}

For the remainder of this section we abbreviate $\delta \deq \tau/10$.

\begin{definition}
A \emph{control parameter} is a deterministic continuous function $\phi \col \f S \to [N^{-1},N^\delta]$ such that $\eta \mapsto \phi(E + \ii \eta)$ is decreasing for each $E$.
\end{definition}
For a control parameter $\phi$ we define
\begin{equation} \label{def_h}
h(\phi) \deq \sqrt{\frac{\im m + \phi}{N \eta}}\,.
\end{equation}
The following result is a simple consequence of Theorem \ref{thm:main_general} applied to the special case of Wigner matrices.

\begin{lemma} \label{lem:main_est_Wig}
Suppose that the assumptions of Theorem \ref{thm:Wigner} hold. Let $\phi$ be a control parameter. Let $z \in \f S$ and suppose that $G - M = O_\prec(\phi)$ at $z$.
\begin{enumerate}
\item
We have at $z$
\begin{equation*}
\Pi(G) = O_\prec \pb{(1 + \phi)^3 h(\phi)}\,.
\end{equation*}
\item
For any deterministic $B \in \C^{N \times N}$ satisfying $\norm{B} = O(1)$ we have at $z$
\begin{equation*}
	\underline{B\Pi(G)}=O_{\prec}\pb{(1+\phi)^{6} h(\phi)^2}\,.
\end{equation*}
\end{enumerate}
\end{lemma}
\begin{proof}
To prove (i), by Lemmas \ref{lem:m_eta} and \ref{lem:imM}, it suffices to prove
\begin{equation} \label{s_G_comp}
\cal S(G) G - g G = O_\prec\pb{(1 + \phi) h(\phi)} \,, \qquad \cal S(M) M - m M = O_\prec\pb{(1 + \phi) h(\phi)}\,,
\end{equation}
where $\cal S$ denotes the map from \eqref{def_S_gen}.
By \eqref{eg} we have $\cal S(G)G=\cal J+\cal K$, and $\eqref{egWard}$ shows $\cal J =O_{\prec} ((1+\phi)\zeta)=O_{\prec}((1+\phi)h(\phi))$. Also, Assumption \ref{ass:Wigner} shows $s_{ij}=1+O(\delta_{ij})$, and we obtain from \eqref{eg} that 
\begin{equation*}
|\cal K_{\f v \f w}-gG_{\f v\f w}|=O(1)\cdot\frac{1}{N}\sum_{i}|G_{ii}G_{i \f w}v_i|=O_{\prec}((1+\phi)h(\phi))
\end{equation*}
for any fixed $\f v, \f w \in \bb S$, where in the last step we used \eqref{eq:bound_on_G}, $|v_{i}| \le 1$, and the Ward identity. This proves the first estimate of \eqref{s_G_comp}. The second estimate of \eqref{s_G_comp} is proved similarly.

The proof of (ii) is analogous, using the estimates \eqref{egWard2} and
\begin{equation*}
\begin{aligned}
|\ul{\cal K B}-g\ul{GB}|&=O(1)\cdot \frac{1}{N^2}\sum_{i,j}|G_{ii}G_{ij}B_{ji}|\\&=O_{\prec}(1+\phi)\cdot \Big(\frac{1}{N^2}\sum_{i,j}|G_{ij}|^2\Big)^{1/2} \cdot \Big( \frac{1}{N^2}\sum_{i,j}|B_{ij}|^2\Big)^{1/2}=O_{\prec}((1+\phi)h(\phi)^2)\,,
\end{aligned}
\end{equation*}
where in the last step we used Lemma \ref{lem:Ward_identity} and the estimate $\sum_{i,j}|B_{ij}|^2=\tr BB^{*} =O(N)$, and Lemmas \ref{lem:m_eta} and \ref{lem:imM}. This concludes the proof.
\end{proof}

For $u \in \C$ define
\begin{equation*}
R_u \deq (A - u - z)^{-1}\,.
\end{equation*}
With this notation, we have
\begin{equation} \label{emakko}
M = R_m \,, \qquad G = R_g - R_g \Pi(G)\,.
\end{equation}
Taking the difference yields
\begin{equation} \label{G-M_main}
G - M = R_g - R_m - R_g \Pi(G)\,.
\end{equation}
Moreover, taking the trace of the second identity of \eqref{emakko} yields
\begin{equation} \label{s-m_main}
\pi(g) = - \ul{R_g \Pi(G)}\,,
\end{equation}
where we defined
\begin{equation*}
\pi(u) \equiv \pi(u,z) \deq u - \ul{R_u} = u - \int \frac{\nu(\dd a)}{a - u - z}\,.
\end{equation*}
Note that $\pi(m) = 0$, by \eqref{eq_spectr_A}.
The formulas \eqref{G-M_main} and \eqref{s-m_main} are the main identities behind the following analysis.

In order to estimate the error terms on the right-hand side of \eqref{G-M_main} and \eqref{s-m_main} using Lemma \ref{lem:main_est_Wig}, we need to deal with the fact that the matrix $R_g$ is random.

\begin{definition}
An event $\Omega$ \emph{holds with high probability} if $\ind{\Omega^c} \prec 0$, i.e.\ if $\P(\Omega^c) \leq N^{-D}$ for all $D > 0$ and large enough $N$ depending on $D$.
\end{definition}

\begin{lemma} \label{lem:Rg_est}
Suppose that the assumptions of Lemma \ref{lem:main_est_Wig} hold. Suppose that $\norm{R_g} = O(1)$ with high probability. Then the conclusions (i) and (ii) of Lemma \ref{lem:main_est_Wig} hold with $\Pi(G)$ replaced by $R_g \Pi(G)$.
\end{lemma}
\begin{proof}
Let $\Omega$ an event of high probability such that $\ind{\Omega} \norm{R_g} = O(1)$, and define the set $\cal U \deq \h{g \equiv g(H) \col H \in \Omega} \subset \C$. Since $\abs{g} \leq N$, we find that $\cal U$ is contained in the ball of radius $N$ around the origin. Let $\hat {\cal U}$ be an $N^{-3}$ net of $\cal U$, i.e.\ a set $\hat {\cal U} \subset \cal U$ such that $\abs{\hat {\cal U}} = O(N^8)$ and for each $ u  \in \cal U$ there exists a $ \hat u  \in \hat {\cal U}$ such that $\abs{ u - \hat u } \leq N^{-3}$.

By a union bound, from Lemma \ref{lem:main_est_Wig} (i) it follows that for any $\f v , \f w \in \bb S$
\begin{equation} \label{sup_s_est}
\sup_{ \hat u  \in \hat {\cal U}} \abs{(R_{ \hat u } \Pi(G))_{\f v \f w}} \prec (1 + \phi)^3 h(\phi)\,.
\end{equation}
Writing
\begin{equation*}
\abs{(R_g \Pi(G))_{\f v \f w}} \leq \abs{(R_{\hat g} \Pi(G))_{\f v \f w}} + \absb{\pb{(R_{\hat g} - R_g) \Pi(G)}_{\f v \f w}}
\end{equation*}
and estimating the first term by $O_\prec((1 + \phi)^3 h(\phi))$ using \eqref{sup_s_est} and the second term by $\norm{R_{\hat g} - R_g} \norm{\Pi(G)} = O(\abs{\hat g - g} \norm{\Pi(G)}) = O(N^{-1})$ (since $\norm{\Pi(G)} = O(N^2)$ trivially), we conclude that $R_g \Pi(G) = O_\prec((1 + \phi)^3 h(\phi))$. Here we used that $N^{-1/2} = O(h(\phi))$ by Lemma \ref{lem:m_eta}. The quantity $\ul{B R_g \Pi(G)}$ is estimated analogously. This concludes the proof.
\end{proof}

\subsection{Proof of the local law}

We begin by estimating $g - m$ for large $\eta$.

\begin{lemma} \label{lem:large_eta}
Under the assumptions of Theorem \ref{thm:Wigner}, we have $g - m = O_\prec((N \eta)^{-1})$ for $\eta \geq 2$.
\end{lemma}
\begin{proof}
For $\eta \geq 1$ we have $\norm{G} \leq 1$ and therefore $G - M = O_\prec(1)$. Moreover, we have $\norm{R_g} \leq 1$ and $\im m \leq 1$.
From \eqref{s-m_main}, $\pi(m) = 0$, and Lemma \ref{lem:Rg_est} we therefore find
\begin{equation*}
g - m = (g - m) \int  \frac{\nu(\dd a)}{(a - g - z) (a - m - z)} + O_\prec\pbb{\frac{1}{N \eta}}\,.
\end{equation*}
For $\eta \geq 2$ we have $\abs{a - g - z} \abs{a - m - z} \geq \im (a - g - z) \im (a - m - z) \geq 4$, and the claim follows since $\nu$ is a probability measure.
\end{proof}

\begin{lemma} \label{lem:G-M_iter}
Suppose that at $z \in \f D$ we have $\norm{M} = O(1)$, $G - M = O_\prec(N^\delta)$, and $g - m = O_\prec(\theta)$ for some control parameter $\theta \leq N^{-\delta}$. Then we have at $z$
\begin{equation*}
G - M = O_\prec(\theta + \psi)\,, \qquad \psi \deq \sqrt{\frac{\im m}{N \eta}} + \frac{1}{N \eta}\,.
\end{equation*}
\end{lemma}
\begin{proof}
From the resolvent identity $R_g = R_m + R_g R_m (g - m)$ and the assumptions $\abs{g - m} \prec N^{-\delta}$ and $\norm{R_m}=\norm{M} = O(1)$, we conclude that $\norm{R_g} = O(1)$ with high probability.

Now suppose that $G - M = O_\prec(\phi)$ for some control parameter $\phi$. Then from \eqref{G-M_main} and Lemma \ref{lem:Rg_est} we find
\begin{equation*}
G - M = R_g R_m (g - m) + O_\prec \pb{(1 + \phi)^3 h(\phi)}\,.
\end{equation*}
Using that $\norm{R_g} = O(1)$ with high probability, we deduce the implication
\begin{equation} \label{self-improv}
G - M = O_\prec(\phi) \qquad \Longrightarrow \qquad G - M = O_\prec\pb{\theta + (1 + \phi)^3 h(\phi)}\,.
\end{equation}
The implication \eqref{self-improv} is a self-improving bound, which we iterate to obtain successively better bounds on $G - M$. The iteration gives rise to a sequence $\phi_0, \phi_1, \dots$, where $\phi_0 \deq N^\delta$ and $\phi_{k+1} \deq \theta + (1 + \phi_k)^3 h(\phi_k)$. By the definition of $\prec$, we easily conclude the claim after a bounded number of iterations. Note that here we used the fact that $\delta = \tau/10$ in order to make sure that $\phi_1 = O(1)$, and for Theorem \ref{thm:main_general}, and thus Lemma \ref{lem:main_est_Wig}, to hold. See also Lemma \ref{lem:self-improv} and its proof.
\end{proof}

Lemma \ref{lem:G-M_iter} provides a bound on $G - M$ starting from a bound on $g - m$ and a rough bound on $G - M$. The needed bound on $g - m$ is obtained from a stochastic continuity argument, which propagates the bound on $g - m$ from large values of $\eta$ to small values of $\eta$. It makes use of the following notion of stability of the equation \eqref{eq_spectr_A}.

\begin{definition}[Stability] \label{def:stability}
Let $\f S \subset \f D$ be a spectral domain. We say that \eqref{eq_spectr_A} is \emph{stable on $\f S$} if the following holds for some constant $C > 0$. Suppose that the line $\cal L \deq \{E\} \times [a,b]$ is a subset of $\f S$. Let $\xi \col \cal L \to [N^{-1}, N^{-\delta/2}]$ be continuous such that $\eta \mapsto \xi(E + \ii \eta)$ is decreasing on $[a,b]$. Suppose that
$\abs{\pi(g)} \leq \xi$
for all $z \in \cal L$, and
\begin{equation} \label{stab2}
\abs{g - m} \leq \frac{C \xi}{\im m + \sqrt{\xi}}
\end{equation}
for $z = E + \ii b$. Then \eqref{stab2} holds for all $z \in \cal L$.
\end{definition}

For the following we fix $E \in \R$, and establish \eqref{G-M_Wig} and \eqref{g-m_Wig} for $z \in (\{E\} \times \R) \cap \f S = \bigcup_{k = 0}^K \cal L_k$, where $\cal L_0 \deq (\{E\} \times [2,\infty))\cap \f S$ and $\cal L_k \deq (\{E\} \times [2 N^{-\delta k}, 2 N^{- \delta (k-1)}]) \cap \f S$ for $k = 1, \dots, K$. Here $K \leq 1/\delta$.

The stochastic continuity argument is an induction on $k$.
We start with $k = 0$. By Lemmas \ref{lem:large_eta} and \ref{lem:G-M_iter} with $\theta = (N \eta)^{-1}$ we have \eqref{G-M_Wig} and \eqref{g-m_Wig} for $z \in \cal L_0$, since $G - M = O_\prec(1)$ trivially.

The induction step follows from the following result.

\begin{lemma}
Fix $k = 1, \dots, K$ and suppose that \eqref{G-M_Wig} and \eqref{g-m_Wig} hold for all $z \in \cal L_{k - 1}$. Then \eqref{G-M_Wig} and \eqref{g-m_Wig} hold for all $z \in \cal L_{k}$.
\end{lemma}
\begin{proof}
Denote by $b_k \deq E + \ii 2 N^{-\delta (k-1)}$ the upper edge of the line $\cal L_k$.
First we note that by a simple monotonicity of the resolvent (see e.g.\ \cite[Lemma 10.2]{BK16}), the estimate $G = O_\prec(1)$ for $z  = b_k$ implies $G = O_\prec(N^\delta)$ for $z \in \cal L_{k}$. Setting $\phi \deq N^\delta$ in Lemma \ref{lem:Rg_est}, we find from \eqref{s-m_main} that $\abs{\pi(g)} \prec N^{7 \delta} (N \eta)^{-1}$ for all $z \in \cal L_k$. Using that $\pi(g)$ is $N^2$-Lipschitz in $\f D$, we find using a simple $N^{-3}$-net argument that with high probability we have $\abs{\pi(g)} \leq (N \eta)^{-1/2}$ for all $z \in \cal L_k$; here we also used that $N^{8 \delta} (N \eta)^{-1} \leq (N \eta)^{-1/2}$ by definition of $\delta$. Moreover, from \eqref{g-m_Wig}, we find that with high probability we have $\abs{g - m} \leq (N \eta)^{-1/2}$ for $z = b_k$. From Definition \ref{def:stability} we therefore deduce that
\begin{equation} \label{g-m_init}
\abs{g - m} \prec (N \eta)^{-1/4} \qquad \text{for all } z \in \cal L_k
\end{equation}

Next, suppose that $\theta \leq N^{-\delta}$ is a control parameter and $\abs{g - m} \prec \theta$ for $z \in \cal L_k$. In particular, from $R_g = R_m + R_g R_m (g - m)$ we deduce that $\norm{R_g} = O(1)$ with high probability for all $z \in \cal L_k$. From Lemma \ref{lem:G-M_iter} we deduce that $G - M = O_\prec(\theta + \psi)$. By Lemma \ref{lem:Rg_est} with $\phi = \theta + \psi$ and \eqref{s-m_main} we therefore have
\begin{equation} \label{F_est_induction}
\abs{\pi(g)} \prec \xi \deq \frac{\im m}{N \eta} + \frac{1}{(N \eta)^2} + \frac{\theta}{N \eta}\,.
\end{equation}
It is easy to verify that $\xi$ is a control parameter.
Moreover, by the induction assumption we have
\begin{equation*}
\abs{g - m} \prec \frac{1}{N \eta}  \leq \frac{C \xi}{\im m + \sqrt{\xi}}\,.
\end{equation*}
From an $N^{-3}$-net argument on $\cal L_k$ analogous to the one given in the previous paragraph, we conclude using Definition \ref{def:stability} that for all $z \in \cal L_k$ we have
\begin{equation*}
\abs{g - m} \prec \frac{\xi}{\im m + \sqrt{\xi}} \leq \frac{C}{N \eta}  + \sqrt{\frac{\theta}{N \eta}}\,.
\end{equation*}
In conclusion, for every $z \in \cal L_k$ we have the implication
\begin{equation*}
\abs{g - m} \prec \theta \qquad \Longrightarrow \qquad \abs{g - m} \prec \frac{1}{N \eta}  + \sqrt{\frac{\theta}{N \eta}}\,.
\end{equation*}
Starting from \eqref{g-m_init} and iterating this implication a bounded number of times and using the definition of $\prec$, we obtain \eqref{g-m_Wig}. See Lemma \ref{lem:self-improv} and its proof. Now \eqref{G-M_Wig} follows from Lemma \ref{lem:G-M_iter}. This concludes the proof.
\end{proof}

We have established the claim of Theorem \ref{thm:Wigner} on the whole line $(\{E\} \times \R) \cap \f S$. Since $E \in \R$ was arbitrary, this concludes the proof of Theorem \ref{thm:Wigner}.

\appendix 
\section{Proof of Lemma \ref{lem:cumulant_expansion}} \label{A}
In this section we prove the cumulant expansion formula with the remainder term \eqref{R_l+1}. We start with an elementary inequality.
\begin{lemma} \label{jeson'}
	Let $X$ be a nonnegative random variable with finite moments. Then for any $a,b,t \ge 0$, we have
	\begin{equation*}
	\E X^a \E \big[X^b \mathbf{1}_{X>t}\big] \le \E \big[X^{a+b}\mathbf{1}_{X>t}\big]\,.
	\end{equation*}
\end{lemma}
\begin{proof}
	It suffices to assume $a>0$. Let us abbreviate $\|X\|_{a}\deq (\E X^a)^{1/a}$. For $t \ge \|X\|_a$, we have
	\begin{equation} \label{1A}
	\E X^a \E \big[X^b \mathbf{1}_{X>t}\big]\le \E \big[ t^aX^b \mathbf{1}_{X>t}\big]\le \E \big[X^{a+b}\mathbf{1}_{X>t}\big]\,,
	\end{equation}
	 which is the desired result.
	For $t < \|X\|_a$, we have
	\begin{equation} \label{2A}
	\E X^a \E \big[X^b \mathbf{1}_{X\le t}\big]>\E \big[t^aX^b \mathbf{1}_{X\le t}\big] \ge \E \big[X^{a+b}\mathbf{1}_{X\le t}\big]\,.
	\end{equation} 
	 Jensen's (or H\"older's) inequality yields
	\begin{equation} \label{3A}
	\E X^a \, \E X^b \le \E X^{a+b}\,,
	\end{equation}
	and the claim follows from \eqref{2A} and \eqref{3A}, using $1 = \mathbf{1}_{X \leq t} + \mathbf{1}_{X > t}$.
\end{proof}
Let $\chi(t)\deq \log \bb E e^{\mathrm{i}t h}$. For $n \ge 1$, we have
\begin{equation*}
\partial^n_t\big(e^{\chi(t)}\big)=\partial^{n-1}_t\big(\chi'(t)e^{\chi(t)}\big)=\sum_{k=1}^n  {n-1 \choose k-1} \partial^k_{t}\big(\chi(t)\big) \partial^{n-k}_{t} \big(e^{\chi(t)}\big)\,, 
\end{equation*} 
hence
\begin{equation*} 
\bb E h^n = (-\ii)^n \partial^n_{t} e^{\chi(t)}\big{|}_{t=0}= \sum\limits_{k=1}^n {n-1 \choose k-1} \mathcal{C}_k(h) \bb E h^{n-k}\,.
\end{equation*}
For $g(h)=h^\ell$, we have
\begin{equation*} 
\E \big[h\cdot g(h)\big] =\E h^{\ell+1} = \sum\limits_{k=1}^{\ell+1} {\ell \choose k-1}\mathcal{C}_k(h) \E h^{\ell+1-k} = \sum\limits_{k=0}^{\ell} \frac{1}{k!} \mathcal{C}_{k+1}(h) \E \qb{g^{(k)}(h)}\,,
\end{equation*}
and by linearity the same relation holds for any polynomial $P$ of degree $\le \ell$:
\begin{equation} \label{A2}
\E \big[h\cdot P(h)\big]=\sum_{k=0}^{\ell} \frac{1}{k!}\cal C_{k+1}(h)\E \big[ P^{(k)}(h)\big]\,.
\end{equation}

Next, let $f$ be as in the statement of Lemma \ref{lem:cumulant_expansion}, and fix $\ell \in \N$. By Taylor expansion we can find a polynomial $P$ of degree at most $\ell$, such that for any $0 \le k \le \ell$,
\begin{equation} \label{A1}
f^{(k)}(h)=P^{(k)}(h)+\frac{1}{(\ell+1-k)!}f^{(\ell+1)}(\xi_k)h^{\ell+1-k}\,,
\end{equation}
where $\xi_k\equiv\xi_k(h)$ is a random variable taking values between 0 and $h$.

By \eqref{A2}, \eqref{A1}, homogeneity of the cumulants, and Jensen's inequality we find that the error term in \eqref{eq:cumulant_expansion} satisfies
\begin{equation} \label{A3}
\begin{aligned}
\cal R_{\ell+1}&=
\mathbb{E}\big[h\cdot f(h)\big]-\sum_{k=0}^{\ell}\frac{1}{k!}\mathcal{C}_{k+1}(h)\mathbb{E}\big[f^{(k)}(h)\big]\\
&=\mathbb{E}\big[h\cdot (f(h)-P(h))\big]-\sum_{k=0}^{\ell}\frac{1}{k!}\mathcal{C}_{k+1}(h)\mathbb{E}\big[f^{(k)}(h)-P^{(k)}(h)\big]\\
&=\frac{1}{(\ell+1)!}\E\big[f^{(\ell+1)}(\xi_0)\cdot h^{\ell+2}\big]-\sum_{k=0}^{\ell}\frac{1}{k!(\ell+1-k)!}\mathcal{C}_{k+1}(h)\mathbb{E}\big[f^{(\ell+1)}{(\xi_k)}\cdot h^{\ell+1-k}\big]\\
&\le O(1)\cdot \sum_{k=0}^{\ell+1} \E |h|^k \cdot \E\Big[\sup_{|x|\le |h|}\big|f^{(\ell+1)}(x)\big|\cdot h^{\ell+2-k} \Big]\\
&\le O(1)\cdot \E |h|^{\ell+2}\cdot \sup_{|x|\le t}\big|f^{(\ell+1)}(x)\big| + O(1)\cdot \sum_{k=0}^{\ell+1} \E |h|^k \cdot \E\Big[\sup_{|x|\le |h|}\big|f^{(\ell+1)}(x)\big|\cdot h^{\ell+2-k} \cdot \mathbf{1}_{|h|>t}\Big]\,.
\end{aligned}
\end{equation}
The desired result then follows from estimating the last term of \eqref{A3} by Cauchy-Schwarz inequality and Lemma \ref{jeson'}.

\section{The complex case} \label{sec:complex}

In this appendix we explain how to generalize our results to the complex case. We briefly explain how to conclude the proof of Theorem \ref{thm:main_general} (i); Theorem \ref{thm:main_general} (ii) follows in a similar fashion.

If $H$ is complex Hermitian, in general $\bb E H^2_{ij}$ and $\bb E |H_{ij}|^2$ are different. We define, in addition to \eqref{31},
\begin{equation}
t_{ij}\deq (1+\delta_{ij})^{-1}N\bb E H^{2}_{ij}\,,
\end{equation}
and, in addition to \eqref{eq:tilted_vectors}, for a vector $\f x=(x_i)_{i \in \bb N} \in  \bb C^N$,
\begin{equation}
\ul{\f x}^{j} = (\ul{x}^{j}_i)_{i\in \qq{N}} \,, \qquad \text{ where } \quad \ul{x}^{j}_i \deq x_i {t}_{ij}\,.
\end{equation}
Now \eqref{eg} generalizes to
\begin{equation} \label{eg2}
(\cal S(G)G)_{\f v\f w} = \mathcal J_{\f v \f w} + \mathcal K_{\f v \f w} \,, \qquad \mathcal J_{\f v \f w} \deq \frac{1}{N}\sum_{j} G_{j \overline{{\ul{\f v}^{j}}}}G_{j\f w}\,, \qquad \mathcal K_{\f v \f w} \deq \frac{1}{N}\sum_{j} G_{jj} G_{ \f v^{j}\f w}\,,
\end{equation} 
and $\cal D_{\f v\f w}$ remains the same as in \eqref{def_D}. 

Following a similar argument as in Section \ref{subsec:preliminarties_to_proofs}, we see that Proposition \ref{prop:est1} is enough to conclude Theorem \ref{thm:main_general} (i). As in \eqref{eq:proof_of_sc_equation_1}, we fix a constant $p \in \bb N$ and write
\begin{equation}\label{eq:proof_of_sc_equation_1'}
\E [|\cal D_{\f v\f w}|^{2p}]	
=\E[\cal K_{\f v\f w}\cal D_{\f v\f w}^{p-1}\overline{\cal D}_{\f v\f w}^p]
+\sum_{i,j}{v}_i\E[H_{ij}G_{j\f w}\cal D_{\f v\f w}^{p-1}\overline{\cal D}_{\f v\f w}^p]\,.
\end{equation}
We compute the second term on the right-hand side of \eqref{eq:proof_of_sc_equation_1'} using the complex cumulant expansion, given in \cite[Lemma 7.1]{HK}, which replaces the real cumulant expansion from Lemma \ref{lem:cumulant_expansion}. The result is
\begin{equation}\label{eq:proof_of_sc_equation_2'}
\E|\cal D_{\f v\f w}|^{2p}
=\E[\cal K_{\f v\f w}\cal D_{\f v\f w}^{p-1}\overline{\cal D}_{\f v\f w}^p]
+\sum_{k=1}^{\ell}\tilde{X}_k+\sum_{i,j}{v}_i  \tilde{\cal R}_{\ell+1}^{ij}\,,
\end{equation}
where $\tilde{X_k}$ and $\tilde{\cal R}^{ij}_{\ell+1}$ are defined analogously to \eqref{X_k} and \eqref{tau1}, respectively. Using the same proof, one can easily extend Lemma \ref{prop:main_estimations} (ii)--(iii) with $X_k$ and $\cal R^{ij}_{\ell+1}$ replaced by $
\tilde{X}_k$ and $\tilde{\cal R}^{ij}_{\ell+1}$, so that it remains to estimate $\E[\cal K_{\f v\f w}\cal D_{\f v\f w}^{p-1}\overline{\cal D}_{\f v\f w}^p]+\tilde{X}_1$. Using the complex cumulant expansion, we find
\begin{equation*}
\begin{aligned}
\tilde{X}_1 &= \sum_{i,j}\Big({v}_{i}\cal \E|H_{ij}^2|\,\E\big[\partial_{ji}\big(G_{j\f w}\cal D_{\f v\f w}^{p-1}\overline{\cal D}_{\f v\f w}^p\big)\big]+{v}_{i}\cal \E[H_{ij}^2]\,\E\big[\partial_{ij}\big(G_{j\f w}\cal D_{\f v\f w}^{p-1}\overline{\cal D}_{\f v\f w}^p\big)\big]\Big)(1+\delta_{ij})^{-1}\\
& = \sum_{i,j}\Big({v}_{i}\cal \E|H_{ij}^2|\,\E\big[G_{j\f w}\partial_{ji}\big(\cal D_{\f v\f w}^{p-1}\overline{\cal D}_{\f v\f w}^p\big)\big]+{v}_{i}\cal \E[H_{ij}^2]\,\E\big[G_{j\f w}\partial_{ij}\big(\cal D_{\f v\f w}^{p-1}\overline{\cal D}_{\f v\f w}^p\big)\big]\Big)(1+\delta_{ij})^{-1}\\ &\ \ \ \ -\mathbb{E}\big[(\mathcal{S}(G)G))_{\mathbf{v}\mathbf{w}}\mathcal{D}_{\mathbf{v}\mathbf{w}}^{p-1}\overline{\mathcal{D}}_{\mathbf{v}\mathbf{w}}^{p}\big]\,,
\end{aligned}
\end{equation*}	
where in the second step we used \eqref{eg2} and
\begin{equation}
\frac{\partial G_{ij}}{\partial H_{kl}}=-G_{ik}G_{lj}\,.
\end{equation}
(Here we take the conventions of \cite[Section 7]{HK} for the derivatives in the complex entries of $H$.)
Thus we have
\begin{equation}\label{eq:estimating_X_1_11}
\begin{aligned}
\E[\cal K_{\f v\f w}\cal D_{\f v\f w}^{p-1}\overline{\cal D}_{\f v\f w}^p]+\tilde{X}_1
=&\, -\E[\cal J_{\f v\f w}\cal D_{\f v\f w}^{p-1}\overline{\cal D}_{\f v\f w}^p]
+ \sum_{i,j}\Big({v}_{i}\cal \E|H_{ij}^2|\,\E\big[G_{j\f w}\partial_{ji}\big(\cal D_{\f v\f w}^{p-1}\overline{\cal D}_{\f v\f w}^p\big)\big]\\
&\ +{v}_{i}\cal \E[H_{ij}^2]\,\E\big[G_{j\f w}\partial_{ij}\big(\cal D_{\f v\f w}^{p-1}\overline{\cal D}_{\f v\f w}^p\big)\big]\Big)(1+\delta_{ij})^{-1}\,.
\end{aligned}
\end{equation}
By estimating the right-hand side of \eqref{eq:estimating_X_1_11} in a similar way as we dealt with the right-hand side of \eqref{eq:estimating_X_1_1}, we obtain $$\E[\cal K_{\f v\f w}\cal D_{\f v\f w}^{p-1}\overline{\cal D}_{\f v\f w}^p]+\tilde{X}_1 = O_\prec (\wt \zeta) \cdot\E|\cal D_{\f v\f w}|^{2p-1} +O_\prec(\wt \zeta \lambda)\cdot\E|\cal D_{\f v\f w}|^{2p-2}\,,$$
which completes the proof.

\bibliography{bibliography} 

\bibliographystyle{amsplain}

\end{document}